\documentclass[11pt,english]{amsart} 

\usepackage{soul}
\usepackage{newtxtext,newtxmath}
\usepackage[T1]{fontenc} 
\usepackage{textcomp} 
\usepackage[english]{babel}
\usepackage{amscd}
\usepackage{url,xspace}
\usepackage{amsbsy,bm} 
\usepackage{paralist}
\usepackage{datetime} 
\usepackage{colortbl}
\usepackage{color}
\usepackage[dvipsnames]{xcolor}
\usepackage{mathrsfs}  
\usepackage{mathtools}
\usepackage{euscript}
\usepackage[all]{xy}
\usepackage{hyperref}
\usepackage{graphicx} 
\usepackage[text={150mm,251mm},centering, marginparwidth=75pt]{geometry} 
\usepackage{comment}  
\usepackage{cite}  
\usepackage{thmtools}
\usepackage{enumitem}
\usepackage{letltxmacro} 
\usepackage{nameref}
\usepackage{cleveref}
\usepackage{calc}
\usepackage{interval}
\usepackage{tikz-cd}  
\usepackage{marginnote}

\usepackage[hyperpageref]{backref}

\usepackage{faktor}

\theoremstyle{plain}
\newtheorem{thmx}{Theorem}
\renewcommand{\thethmx}{\Alph{thmx}} 
\newtheorem{thm}{Theorem}[section]  
\newtheorem{lem}[thm]{Lemma}
\newtheorem{claim}[thm]{Claim} 
\newtheorem{proposition}[thm]{Proposition}
\newtheorem{cor}[thm]{Corollary}

\theoremstyle{definition}
\newtheorem{dfn}[thm]{Definition}

\theoremstyle{remark}
\newtheorem{rem}[thm]{Remark}

\numberwithin{equation}{section}  

\theoremstyle{plain}
\newlist{thmlist}{enumerate}{1}
\setlist[thmlist]{wide = 0pt, labelwidth = 2em, labelsep*=0em, itemindent = 0pt, leftmargin = \dimexpr\labelwidth + \labelsep\relax, noitemsep,topsep = 1ex, font=\normalfont, label=(\roman*), ref=\thethm.(\roman{thmlisti})}

\addtotheorempostheadhook[thm]{\crefalias{thmlisti}{thm}}

\addtotheorempostheadhook[assumpsion]{\crefalias{thmlisti}{assumption}}

\addtotheorempostheadhook[cor]{\crefalias{thmlisti}{cor}}

\addtotheorempostheadhook[proposition]{\crefalias{thmlisti}{proposition}}

\addtotheorempostheadhook[dfn]{\crefalias{thmlisti}{dfn}}

\addtotheorempostheadhook[lem]{\crefalias{thmlisti}{lem}}
\addtotheorempostheadhook[main]{\crefalias{thmlisti}{main}}

\addtotheorempostheadhook[rem]{\crefalias{thmlisti}{rem}}

\newlist{thmenum}{enumerate}{1} 
\setlist[thmenum]{wide = 0pt, labelwidth = 2em, labelsep*=0em, itemindent = 0pt, leftmargin = \dimexpr\labelwidth + \labelsep\relax, noitemsep,topsep = 1ex, font=\normalfont, label=(\roman*), ref=\thethmx.(\roman{thmenumi})}
\crefalias{thmenumi}{thmx} 

\newlist{corlist}{enumerate}{1} 
\setlist[corlist]{wide = 0pt, labelwidth = 2em, labelsep*=0em, itemindent = 0pt, leftmargin = \dimexpr\labelwidth + \labelsep\relax, noitemsep,topsep = 1ex, font=\normalfont, label=(\roman*), ref=\thecorx.(\roman{corlisti})}
\crefalias{corlisti}{corx} 

\crefname{lem}{Lemma}{Lemmas} 
\crefname{conjecture}{Conjecture}{Conjectures}
\crefname{thm}{Theorem}{Theorems}
\crefname{proposition}{Proposition}{Propositions}
\crefname{dfn}{Definition}{Definitions}
\crefname{rem}{Remark}{Remarks}
\crefname{cor}{Corollary}{Corollaries}
\crefname{corx}{Corollary}{Corollaries}
\crefname{problem}{Problem}{Problems}
\crefname{thmx}{Theorem}{Theorems}
\crefname{claim}{Claim}{Claims}
\crefname{assumption}{Assumption}{Assumptions}
\crefname{main}{Main Theorem}{Main Theorems}

\def\ep{\varepsilon}

\def\pr{{\rm pr}}

\newcommand{\cR}{\mathcal{R}}
\newcommand{\cS}{\mathcal{S}}

\makeatletter
\newcommand*{\rom}[1]{\expandafter\@slowromancap\romannumeral #1@}
\makeatother

\setlist[itemize]{wide = 0pt, labelwidth = 2em, labelsep*=0em, itemindent = 0pt, leftmargin = \dimexpr\labelwidth + \labelsep\relax, noitemsep,topsep = 1ex,}
\setlist[enumerate]{wide = 0pt, labelwidth = 2em, labelsep*=0em, itemindent = 0pt, leftmargin = \dimexpr\labelwidth + \labelsep\relax, noitemsep,topsep = 1ex}

\makeatletter     
\newcommand{\crefnames}[3]{%
	\@for\next:=#1\do{%
		\expandafter\crefname\expandafter{\next}{#2}{#3}%
	}%
}
\makeatother

\crefnames{part,chapter,section}{\S}{\S\S}

\newcommand{\sD}{\mathscr{D}}

\newcommand{\sO}{\mathscr{O}}

\newcommand{\cC}{\mathcal C}

\newcommand{\cF}{\mathcal F}

\newcommand{\cN}{\mathcal N}

\newcommand{\cO}{\mathcal O}

\newcommand{\bB}{\mathbb{B}}
\newcommand{\bC}{\mathbb{C}}
\newcommand{\bD}{\mathbb{D}}

\newcommand{\bH}{\mathbb{H}}

\newcommand{\bN}{\mathbb{N}}

\newcommand{\bP}{\mathbb{P}}

\newcommand{\bR}{\mathbb{R}}

\newcommand{\bU}{\mathbb{U}}

\newcommand{\bZ}{\mathbb{Z}}

\newcommand{\xsp}{X^{\! \rm sp}}
\newcommand{\zsp}{Z^{\! \rm sp}}
\newcommand{\ysp}{Y^{\! \rm sp}}

\def\db{\bar{\partial}}

 \def\d{\partial} 
\def\hess{\partial\db}

\def\sn{\sqrt{-1}}

\def\Sym{{\text{Sym}}}
\def\vol{\text{\small  Vol}}

\newcommand{\GL}{{\rm GL}}

\begin{document} 
	\title[Pluriharmonic maps to Euclidean buildings]{Existence and Unicity of pluriharmonic maps to Euclidean buildings and applications} 
	
	\date{\today} 
\author[Y. Deng]{Ya Deng}
\address{CNRS, Institut \'Elie Cartan de Lorraine, Universit\'e de Lorraine, Site de
	Nancy,   54506 Vand\oe uvre-lès-Nancy, France} 
\email{ya.deng@math.cnrs.fr}
\urladdr{https://ydeng.perso.math.cnrs.fr} 
	
	\author[C. Mese]{Chikako Mese}
	\address{Johns Hopkins University, Department of Mathematics, Baltimore, MD}
	\email[Chikako Mese]{cmese@math.jhu.edu} 
	\urladdr{https://sites.google.com/view/chikaswebpage/home}

	\begin{abstract}  
		Given a  complex smooth quasi-projective variety $X$, a reductive  algebraic group $G$ defined over some non-archimedean local field $K$ and  a Zariski dense   representation   $\varrho:\pi_1(X)\to G(K)$, we construct  a $\varrho$-equivariant  pluriharmonic map from the universal cover  of $X$ into the Bruhat-Tits building $\Delta(G)$ of $G$, with appropriate asymptotic behavior. We also establish the uniqueness of such a pluriharmonic map in a suitable sense, and provide a geometric characterization of these equivariant maps.  This paper builds upon and extends previous work by the authors  jointly with G.~Daskalopoulos and D.~Brotbek. 
	\end{abstract}

	\maketitle
	\tableofcontents	
	\section{Introduction}
	
	Pluriharmonic maps are a  powerful tool in establishing a bridge between differential geometry, complex geometry, and geometric group theory. 
In his influential paper \cite{siu}, Y.-T.~Siu proved that harmonic maps between K\"ahler manifolds of complex dimension $\geq 2$ are pluriharmonic, provided that the domain manifold is compact and the target is non-positively curved in a strong sense.  With this, Siu proved  the strong rigidity of compact Kähler manifolds, a significant extension of  Mostow's strong rigidity theorem for locally symmetric spaces.  The applications of pluriharmonic maps in the manifold setting have been further developed by various authors, including Sampson \cite{Sam85}, Corlette \cite{Cor88}, Carlson-Toledo \cite{carlson-toledo} and Jost-Yau \cite{jost-yau} and Mochizuki \cite{Moc07},  contributing to the study of diverse problems such as   rigidity of representations, Higgs bundles,  higher Teichm\"uller space, and non-abelian Hodge theory.
 Gromov-Schoen \cite{GS92} extended these ideas by studying pluriharmonic maps into Euclidean buildings.  Building on this work,  we continue to extend Gromov and Schoen's ideas to a more general setting where the domain is a quasiprojective variety (cf.~\cite{BDDM}).   The results of this paper play a crucial role in subsequent papers, for example \cite{DW24b} and \cite{DMW24}.  Furthermore, these findings have potential applications in the study of rigidity phenomena for the mapping class group of surfaces, which will be explored in future work. 

	\subsection{Main results}
	We first establish the following existence theorem of equivariant pluriharmonic maps to Bruhat-Tits buildings. 
	\begin{thmx}\label{main:harmonic}
		Let $X$ be a smooth quasi-projective variety and let $G$ be a reductive group defined over a non-archidemean local field $K$. Let $\Delta(G)$ be  the \emph{enlarged} Bruhat-Tits building of $G$. Denote by  $\pi_X:\widetilde{X}\to X$ the universal covering map.   If $\varrho:\pi_1(X)\to G(K)$ is a Zariski dense representation, then the following statements hold:
		\begin{thmenum}
			\item  \label{item:existence}  There exists a $\varrho$-equivariant pluriharmonic map $\tilde{u}:\widetilde{X}\to \Delta(G)$ with  logarithmic energy growth. 
			\item \label{item:harmonic}$\tilde{u}$ is harmonic with respect to any K\"ahler metric on $\widetilde{X}$.  	
			\item  \label{item:pullback}Let $f: Y \to X$ be a morphism from a smooth quasi-projective variety $Y$. Denote by $\tilde{f}: \widetilde{Y} \to \widetilde{X}$   the lift of $f$ between the universal covers of $Y$ and $X$.  Then the $f^*\varrho$-equivariant map $\tilde{u }\circ \tilde{f}: \widetilde{Y} \to \Delta(G)$ is   pluriharmonic and has logarithmic energy growth. 
			\item \label{item:singular}There is a proper Zariski closed subset $\Xi$ of $X$ such that the singular set $\cS(\tilde{u})$ of $\tilde{u}$ defined in \cref{def:sing} is contained in $\pi_X^{-1}(\Xi)$.  
		\end{thmenum}
	\end{thmx} 
	Note that  when $X$ is a compact K\"ahler manifold,  \cref{item:existence,item:harmonic,item:pullback} were established by Gromov-Schoen in \cite{GS92}  and \cref{item:singular} was proved by Eyssidieux in \cite{Eys04}.  In the case  where $G$ is semisimple, \cref{item:existence,item:harmonic,item:pullback} were proven by the authors with   Brotbek and Daskalopoulos  in \cite[Theorem A]{BDDM}. 
	
	In general, the uniqueness of the equivariant pluriharmonic map in \cref{main:harmonic} is not guaranteed, although it can be established under additional assumptions on the representation (cf.~\cite{DMunique,BDDM}). However, we   prove the uniqueness in a suitable local setting over a dense open subset of $\widetilde{X}$ that has full Lebesgue measure.  
	\begin{thmx}\label{main:unicity}
		Let $X$ be a smooth quasi-projective variety and let $G$ be a reductive group defined over a non-archidemean local field $K$.  For a representation $\varrho:\pi_1(X)\to G(K)$, if  $\tilde{u}_0,\tilde{u}_1:\widetilde{X}\to \Delta(G)$ are  two $\varrho$-equivariant pluriharmonic maps with  logarithmic energy growth, then for \emph{almost every} point $x\in \widetilde{X}$, it has an open neighborhood   $\Omega$ such that 
		\begin{thmenum}
			\item there exists  an apartment $A$ of $\Delta(G)$ that contains both $\tilde{u}_0(\Omega)$ and $\tilde{u}_1(\Omega)$;
			\item  \label{item:translation}  The map $  \tilde{u}_0|_{\Omega}:\Omega\to A$ is a translate of $  \tilde{u}_1|_{\Omega}:\Omega\to A$.  
		\end{thmenum}
	\end{thmx}
	It is worthwhile mentioning that \cref{main:harmonic,main:unicity} were established by Corlette-Simpson \cite{CS08} for the case where $G={\rm PSL}_2$ and the representation  $\varrho$ has quasi-unipotent monodromies  at infinity.  In this setting,  the Bruhat-Tits building of $G$ is a tree, and the energy of the $\varrho$-equivariant pluriharmonic map $\tilde{u}$ is proven to be finite. 
	
	Finally, we   give a geometric  characterization of pluriharmonic maps with logarithmic energy growth in terms of spectral covers. 
	\begin{thmx}\label{main:spectral}
		Let  $X$, $\varrho$, $G$ and $\tilde{u}$ be as in \cref{main:harmonic}. Let $\overline{X}$ be a smooth projective compactification of $X$ such that $\Sigma:=\overline{X}\backslash X$ is a simple normal crossing divisor.   Then  we have  the following:
		\begin{thmenum}
			\item The $\varrho$-equivariant pluriharmonic map $\tilde{u}$ induces a multivalued logarithmic 1-form $\eta$ on the log pair $(\overline{X},\Sigma)$, satisfying the properties in \cref{item:ortho}. 
			\item \label{item:unicity}Such $\eta$ does not depend on the choice of $\tilde{u}$; i.e., if $\tilde{v}$ is another $\varrho$-equivariant pluriharmonic map with logarithmic energy growth, the multivalued logarithmic 1-form induced by $\tilde{v}$ is $\eta$.  
			\item There exists a ramified Galois cover  $\pi:\overline{\xsp}\to \overline{X}$   such that  $\pi^*\eta$ becomes single-valued; i.e., $\pi^*\eta= \{\omega_1,\ldots,\omega_m\}\subset H^0(\overline{\xsp},\pi^*\Omega_{ \overline{X}}(\log \Sigma))$. 
			\item \label{item:pure} Denote by   $\Sigma_1:=\overline{\xsp}\backslash \xsp$.  Let $\mu:\overline{Y}\to \overline{\xsp}$ be a  log resolution of $(\overline{\xsp},\Sigma_1)$, with $\Sigma_Y:=\mu^{-1}(\Sigma_1)$ a simple normal crossing divisor.   Then   $\{\mu^*\omega_1,\ldots,\mu^*\omega_m\}\in H^0(\overline{Y},\Omega_{ \overline{Y}}(\log \Sigma_Y))$ are  \emph{pure imaginary}, i.e.,  the residue of every $
			\mu^*\omega_j
			$  at    each irreducible component  of $\Sigma_Y$  is a pure imaginary number.  
		\end{thmenum}
	\end{thmx} 
	We mention   that \cref{item:pure}
	is analogous to Mochizuki’s notion of   \emph{pure imaginary harmonic bundles}  induced by pluriharmonic maps to symmetric spaces associated with complex semisimple local systems over quasi-projective varieties (cf. \cite{Moc07}).  In our case, however, a spectral cover is required to transform the multivalued logarithmic 1-form induced by the pluriharmonic map into  single-valued logarithmic  1-forms. 
	 
	 In this paper, we assume that $K$ is a non-archimedean local field endowed with a \textit{discrete} non-archimedean valuation. On the other hand, a recent paper by C.~Breiner, B.~Dees, and the second author \cite{BDM24} introduces techniques to study the case for a general non-archimedean valuation local field $L$. In the sequel, we will show that \cref{main:harmonic}, \cref{main:unicity}, and \cref{main:spectral} generalize to the case where $\Delta(G)$  is a Bruhat-Tits building associated with a reductive group defined over any non-archimedean  field. 
	\subsection{Notation and Convention}\label{sec:notation}
	\begin{enumerate}[label=(\alph*)]
		\item Unless otherwise specified,  algebraic varieties are assumed to be connected and   defined over the field of complex 
		numbers.
		\item Let $G$ be a reductive group defined over a non-archimedean local field $K$. We denote by $\Delta(G)$ the Bruhat-Tits building of $G$, which is a non-positively curved (NPC for short) space. Denote by  $d(\bullet,\bullet)$ the distance function on $\Delta(G)$. Denote by $\sD G$ the derived group of $G$, which is semisimple. 
		\item For a complex space $X$, denote by $X^{\rm norm}$ the normalization of $X$.
		\item   A \emph{log smooth pair} $(\overline{X},\Sigma)$ consists of a   smooth projective variety $\overline{X}$ and a simple normal crossing divisor $\Sigma$.  We  denote by $X:=\overline{X}\backslash \Sigma$, and    $\pi_X:\widetilde{X}\to X$ the universal covering map. 
		\item Say a function $\tilde{f}$  (resp. a 1-form $\tilde{\eta}$) on $\widetilde{X}$  descends on $X$ if there exists a function $ {f}$ (resp. a 1-form $ {\eta}$) on $ {X}$   such that $\tilde{f}=\pi_X^*f$ (resp. $\tilde{\eta}=\pi_X^*\eta$).
		\item Let $\overline{X}$ be a smooth projective variety. A line bundle $L$ on $\overline{X}$ is \emph{sufficiently
			ample} if there exists a projective embedding $\iota:\overline{X}\hookrightarrow \bP^N$  such that $L=\iota^*\sO_{\bP^N}(d)$ for some $d\geqslant 3$.
		\item A linear representation $\varrho:\pi_1(X)\to \GL_N(K)$ with $K$ some field is called \emph{reductive} if the Zariski closure of $\varrho(\pi_1(X))$ is a reductive algebraic group over $\overline{K}$.
		
		If
		$Y$ is a closed smooth subvariety of $X$, we denote by $\varrho_Y:\pi_1(Y)\to G(K)$    the composition of the natural homomorphism $\pi_1(Y)\to \pi_1(X)$ and $\varrho$.  
		\item Denote by $\bD$ the unit disk in $\bC$, and by $\bD^*$ the punctured unit disk. We write $\bD_r := \{ z \in \bC \mid |z| < r \}$, $\bD^*_r := \{ z \in \bC \mid 0 < |z| < r \}$, and $\bD_{r_1,r_2} := \{ z \in \bC \mid r_1 < |z| < r_2 \}$. 
	\end{enumerate} 
	\subsection*{Acknowledgment} The first author is supported in part by    ANR-21-CE40-0010.	The second author is  supported in part by NSF DMS-2304697.  Part of this research was performed while the second author was visiting the Mathematical Sciences Research Institute (MSRI), now becoming the Simons Laufer Mathematical Sciences Institute (SLMath), which is supported by the National Science Foundation (Grant No. DMS-1928930).     Finally, the second author thanks the Mathematical Society of Japan and the organizers of the 14th MSJ-SI: New Aspects of Teichmüller Theory for the invitation to their conference at the University of Tokyo in July 2022.

	\section{Technical preliminary}
		For more details of this section, we refer the readers to \cite{BDDM}.
	\subsection{NPC spaces and Euclidean buildings}
	For more details, we refer the readers to \cite{bridson-haefliger,Rou09,KP23}. 
	\begin{dfn}[Geodesic  space] Let $(Z,d_Z)$ be a metric space.  A curve $\gamma:[0, \ell] \rightarrow Z$ into $Z$ is called a geodesic if the length $d_Z(\gamma(a),\gamma(b))=b-a$ for any subinterval $[a, b] \subset[0,\ell]$.  A metric space $(Z,d_Z)$ is a \emph{geodesic space} if there exists a geodesic connecting every pair of points in $Z$.
	\end{dfn}
	\begin{dfn}[NPC space]An NPC (non-positively curved) space $(Z,d_Z)$ is a complete geodesic space that satisfies the following condition: for any three points $P,Q,R\in Z$ and a  geodesic $\gamma:[0, \ell] \rightarrow Z$ with $\gamma(0)=Q$ and $\gamma(\ell)=R$, we have
		$$
		d^{2}\left(P, Q_{t}\right) \leq(1-t) d^{2}(P, Q)+t d^{2}(P, R)-t(1-t) d^{2}(Q, R)
		$$
		for any $t\in [0,1]$, where $Q_{t}:=\gamma(t\ell)$.
	\end{dfn}
	A smooth Riemannian manifold with nonpositive sectional curvature is a familiar example of an NPC space. Of particular interest in this paper are the Bruhat-Tits buildings, $\Delta(G)$, associated with semisimple algebraic groups $G$ defined over non-Archimedean local fields $K$. While we do not provide a detailed description of Bruhat-Tits buildings here, we refer interested readers to \cite{Rou09} and \cite{KP23} for precise definitions.
	
	We highlight two key properties of Bruhat-Tits buildings. First, an apartment of $\Delta(G)$ is a maximal flat, Euclidean subcomplex that is isometric to a Euclidean space $\mathbb{R}^N$, where $N = \dim \Delta(G)$. Moreover, $\Delta(G)$ can be expressed as the union of its apartments. Second, the group of $K$-points of $G$, denoted by $G(K)$, acts isometrically on $\Delta(G)$ and transitively on its set of apartments. The dimension of $\Delta(G)$, in turn, corresponds to the $K$-rank of $G$, which is the dimension of a maximal $K$-split torus in the algebraic group $G$.
	
	\subsection{Harmonic maps to NPC spaces}\label{sec:harmonic}
	Let  $f: \Omega\to Z$ be a map from an  $n$-dimensional bounded Lipschitz Riemannian domain.  $(\Omega, g)$ to an NPC space $(Z, d_Z)$.  When $Z$ is a smooth  Riemannian manifold of nonpositive sectional curvature,  the energy of a smooth map $f: \Omega \rightarrow Z$ is
	$$
	E^{f}=\int_{\Omega}|d f|^{2} d\operatorname{vol}_{g}
	$$
	where  $d \operatorname{vol}_{g}$ is the volume form of $\Omega$.   The assumption of nonpositive sectional curvature ensures the convexity of the energy functional. In particular, this implies that a harmonic  map $f: \Omega\to Z$   is locally energy minimizing.  In other words,  for any $p \in \Omega$, there exists $r>0$ such that the restriction $\left.u\right|_{B_{r}(p)}$ minimizes energy amongst all maps $v: B_{r}(p) \rightarrow Z$ with the same boundary values as $\left.u\right|_{B_{r}(p)}$.
	
	We now recall the notion of harmonic maps in the context of NPC space targets, where the target space is not necessarily a manifold. (cf. \cite{KS} for more details).
	Let $\Omega_{\ep}$ be the set of points in $\Omega$ at a distance least $\ep$ from $\partial \Omega$. Let $B_{\ep}(x)$ be a geodesic ball centered at $x$ and $S_{\ep}(x)=\partial B_{\ep}(x)$. We say $f: \Omega \rightarrow Z$ is an $L^{2}$-map (or that $f \in L^{2}(\Omega, Z)$ ) if for some point $P\in \Omega$, we have  
	$$
	\int_{\Omega} d^{2}(f, P) d \mathrm{vol}_{g}<\infty.
	$$
	For $f \in L^{2}(\Omega, Z)$, define
	$$
	e^f_{\ep}: \Omega \rightarrow \mathbb{R}, \quad e^f_{\ep}(x)= \begin{cases}\int_{y \in S_{\ep}(x)} \frac{d^{2}(f(x), f(y))}{\ep^{2}} \frac{d \sigma_{x, \ep}}{\ep^{n-1}} & x \in \Omega_{\ep} \\ 0 & \text { otherwise }\end{cases}
	$$
	where $\sigma_{x, \ep}$ is the induced measure on $S_{\ep}(x)$. We define a family of functionals
	$$
	E_{\ep}^{f}: C_{c}(\Omega) \rightarrow \mathbb{R}, \quad E_{\ep}^{f}(\varphi)=\int_{\Omega} \varphi e^f_{\ep} d\vol_{g} .
	$$
	We say $f$ has finite energy (or that $f \in W^{1,2}(\Omega, Z)$ ) if
	$$
	E^{f}:=\sup _{\varphi \in C_{c}(\Omega), 0 \leq \varphi \leq 1} \limsup _{\ep \rightarrow 0} E_{\ep}^{f}(\varphi)<\infty .
	$$ 
	In this case,  \cite[Theorem 1.10]{KS} asserts that   there exists an absolutely continuous function $e^f(x)$ with respect to   Lebesgue measure,   which we call the \emph{energy density}, such that $e^f_{\ep}(x) d \mathrm{vol}_{g} $ converges weakly to $ e^f(x) d \operatorname{vol}_{g}$ as $\ep$ tends to $0$. In analogy to the case of smooth targets, we write $|\nabla f|^{2}(x)$ in place of $e^f(x)$. Hence $|\nabla f|^{2}(x)\in L^1_{\rm loc}(\Omega)$. In particular, the (Korevaar-Schoen) energy of $f$ in $\Omega$ is
	\begin{align}\label{eq:defenergy}
		E^{f}[\Omega]=\int_{\Omega}|\nabla f|^{2} d \operatorname{vol}_{g} . 
	\end{align}

	\begin{dfn}[Harmonic maps]
		We say a continuous map $f: \Omega \rightarrow Z$ from a Lipschitz domain $\Omega$ is {\it harmonic} if it is locally energy minimizing; more precisely, at each $p \in \Omega$, there exists  
		an open  neighborhood $\Omega_p$ of $p$ such that all  comparison maps which agree with $u$ outside of this neighborhood have no less energy. 
	\end{dfn}

	\subsection{Equivariant maps and sections} \label{sec:ems}
	Endow $X$ with a K\"ahler metric $g$.  Let  $\varrho :\pi_1(X) \rightarrow G(K)$ be a representation where $G$ is a reductive algebraic group over a non-archimedean local field $K$.    
	The set of all $\varrho $-equivariant maps  into $\Delta(G)$ are in one-to-one correspondence with the set of all sections of the fiber bundle
	$
	\Pi: \widetilde{X} \times_{\varrho } \Delta(G)  \rightarrow X
	$.  More precisely, for a $\varrho $-equivariant map $\tilde{f}: \widetilde{X} \rightarrow \Delta(G)$,  we define a  section of $\Pi$ by  setting
	$
	f(\pi_X(\tilde p)) = [(\tilde p, \tilde f(\tilde p))],
	$ 
	where $\tilde{p}$ is any point in $\widetilde{X}$. We shall use this notation throughout this paper.
	
	One can also define the energy density function  $|\nabla \tilde f|^2$  of $\tilde{f}$, and we refer the readers to \cite{KS,BDDM} for the definition. 
	Since $\tilde{f}$ is equivariant,  $|\nabla \tilde f|^2$ on $\widetilde{X}$ is a $\pi_1(X)$-invariant function, and thus  it descends to a function on $X$, denoted by 
	$
	|\nabla f|^2$. We also define  the energy of $f$ in any open subset $U$ of $X$ by setting
	\begin{align}\label{eq:defenergy2}
		E^f[U] = \int_U |\nabla f|^2 d\mbox{vol}_g. 
	\end{align} 
	\subsection{Pullback bundles}\label{sec:pullback}
	Let $f:Y\to X$ be a morphism between smooth quasi-projective varieties. Let $\cC$ be an NPC space, and let $\varrho:\pi_1(X)\to {\rm Isom}(\cC)$ be a homomorphism.  Let $\widehat{Y}$ be a connected component of $\widetilde{X}\times_XY$.  
	Then we have the following commuting diagram:
	\begin{equation*}
		\begin{tikzcd}\widetilde Y \arrow[d, 
			"\pi_{\widehat Y}"]\arrow[dd, bend right=30, "\pi_Y"']
			&
			\\
			\widehat Y \arrow[r, "\hat{f}"] \arrow[d, 
			"\hat \pi_Y"]
			& \widetilde X \arrow[d, "\pi_X" ] \\
			Y  \arrow[r, "f" ]
			& X  \end{tikzcd}
	\end{equation*}
 Here, $\pi_{\widehat{Y}}$ and $\hat{\pi}_Y$ are covering maps, and $\hat{f}$ is the natural holomorphic map defined as the composition $\widetilde{Y} \hookrightarrow \widetilde{X} \times_X Y \to X$. 
 It induces a fiber bundle $\hat{\Pi}_Y:\widehat{Y}\times_{f^*\varrho } \mathcal C\to Y$, 
	%
	such that one has the following commuting diagram:
	\begin{equation*}
		\begin{tikzcd}
			\widehat Y \times_{f^*\varrho } \mathcal C \arrow[r, "F"] \arrow[d, 
			"\hat \Pi_Y"]
			& \widetilde X \times_{\varrho } \mathcal C \arrow[d, "\Pi_X" ] \\
			Y  \arrow[r, "f" ]
			& X. \end{tikzcd}
	\end{equation*}
	By \cref{sec:ems}, a $\varrho$-equivariant map $\tilde{u}:\widetilde{X}\to \cC$ corresponds to a  section  $u: X \rightarrow  \widetilde X \times_\varrho \mathcal C$ of $\Pi_X$. The composition
	\[
	u \circ f:Y \rightarrow \widetilde X \times_\varrho \mathcal C
	\]
	defines a section  
	of the fiber bundle $\widehat Y \times_{f^*\varrho} \mathcal C \simeq f^*(\widetilde X \times_\varrho \mathcal C)\to Y$, which in turn defines a  $ {f}^*\varrho$-equivariant map $\hat u_f: \hat Y \rightarrow \mathcal C$.  Define $\widetilde{u_f}:=\hat{u}_f\circ \pi_{\widehat{Y}}$, which is an  $f^*\varrho$-equivariant map $\widetilde Y \rightarrow \mathcal C$. It  defines a section
	\begin{equation*} 
		u_f: Y \rightarrow \widetilde Y \times_{f^*\varrho} \mathcal C.
	\end{equation*}
	In this paper, we will mainly  focus on the special case where  $Y$ is a closed smooth subvariety of $X$ and  
	$
	\iota: Y  \rightarrow X
	$ 
	is the inclusion map. In this cases,  we will use the notation
	\begin{align}\label{restrictionsection}
		u_Y: Y \rightarrow \widetilde Y \times_{\varrho_Y} \mathcal C.
	\end{align}
	in place of $u_\iota$, where $\varrho_Y:\pi_1(Y)\to {\rm Isom}(\cC)$ denotes  the composition of $\iota_*:\pi_1(Y)\to \pi_1(X)$ and $\varrho$.  Denote by $\widetilde{u_Y}:\widetilde{Y}\to \cC$ the corresponding $\varrho_Y$-equivariant map.   
\subsection{Regularity results of Gromov-Schoen}

Let $X$ be a hermitian manifold and let  $\tilde{u}:\widetilde{X}\to \Delta(G)$ be a $\varrho$-equivariant harmonic map.  Following  \Cref{sec:pullback},  let  $u:X  \rightarrow \widetilde{X} \times_{\varrho } \Delta(G)$ be the section corresponding to $\tilde{u}$.  We recall some results in \cite{GS92}. 

\begin{dfn}[Locally Lipschitz]
	A continuous map $f:\Omega\to Z$ is called \emph{locally Lipschitz} if for any $p\in \Omega$, there exists   an open  neighborhood $\Omega_p$ of $p$   and a constant $C>0$ such that  $d(f(x), f(y))\leq Cd(x,y)$ for $x,y\in \Omega_p$. 
\end{dfn}
\begin{rem}\label{rem:Litschitz}
It follows from the  definition of $|\nabla f|^2$ that if $f$ is locally Lipschitz, then for any $p\in \Omega$,  there exists    an open  neighborhood $\Omega_p$ of $p$   and a constant $C>0$  such that  over $\Omega_p$ one has
	 $
	 |\nabla f|^2\leq C.
	 $   
\end{rem}
\begin{thm}[\cite{GS92}, Theorem 2.4]\label{thm:Lcon}
	A harmonic map $\tilde{u}:\widetilde X \rightarrow \Delta(G)$ is locally Lipschitz.\qed 
\end{thm}

\begin{dfn}[Regular points and singular points] \label{def:sing}
	A point $x \in \widetilde{X}$ is said to be a {\it regular point} of $\tilde{u}$  if there exists a neighborhood ${\mathcal N}$ of $x$ and an apartment $A \subset \Delta(G)$ such that $\tilde{u}(\mathcal N) \subset A$.  
	A {\it singular point} of $\tilde{u}$ is a point in $\widetilde{X}$ that is not a regular point.   Note that  if $x \in \widetilde{X}$ is a regular point (resp.~singular point) of $\tilde{u}$, then every point of $\pi_X^{-1}(\pi_X(x))$ is a  regular point (resp.~singular point) of $\tilde{u}$.
	We denote by $\mathcal R(\tilde{u})$ (resp.~$\mathcal S(\tilde{u})$) the set of all regular points (resp.~singular points) of $\tilde{u}$ and let $\mathcal R(u)=\pi_X(\mathcal R(\tilde{u}))$ (resp.~$\mathcal S(u)=\pi_X(\mathcal S(\tilde{u}))$). 
\end{dfn}

\begin{lem}[ \cite{GS92}, Theorem 6.4] \label{gs}
	The set ${\mathcal S}(u)$  is  a closed subset of $X$ of  Hausdorff codimension at least two. \qed
\end{lem} 
 
\begin{rem}
B.~Dees \cite{Dee22} improved \cref{gs} to show that ${\mathcal S}(u)$ is $(n-2)$-countably rectifiable where $n$ is the dimension of the domain.
\end{rem} 
 
\subsection{Pluriharmonicity}  

\begin{dfn}[Pluriharmonic maps] \label{def:pluriharmonic}
Let $X$ be  a complex manifold.  A locally Lipschitz map $u:X \rightarrow \Delta(G)$ is said to be \emph{pluriharmonic} if $u \circ \psi$ is harmonic for any holomorphic map $\psi: \bD \rightarrow X$. 
\end{dfn} 
By \cite[Lemma 2.11]{BDDM}, in order to establish the pluriharmonicity of a harmonic map $u$, it is sufficient to verify it over the regular set of $u$. 

\begin{lem}[cf.~\cite{BDDM}] \label{phequiv}
Let $u: U = \mathbb{D}^n \rightarrow \Delta(G)$ be a harmonic map with respect to the standard Euclidean metric on $U = \mathbb{D}^n$.  If  $\partial \bar{\partial} u = 0$ on the regular set $\mathcal{R}(u)$, then $u$ is pluriharmonic as defined in Definition~\ref{def:pluriharmonic}.  
\end{lem}
\begin{rem}
	 Note that if $x \in \mathcal{R}(u)$, we can select a neighborhood $V$ of $x$ and an apartment $A$ so that $u(V) \subset A$. Our assumption implies that, upon identifying $A \simeq \mathbb{R}^N$, the map $u: V \rightarrow \mathbb{R}^N$ is smooth and satisfies $\partial \bar{\partial} u = 0$. 
\end{rem} 
\subsection{Logarithmic energy growth}
Let $X$ be a smooth quasi-projective variety. Let $\cC$ be an NPC space.    Consider a representation $\varrho:\pi_1(X)\to {\rm Isom}(\cC).$ We define: \begin{dfn}[Translation length]\label{def:translation}
	For an element $\gamma\in \pi_1(X)$, 	the \emph{translation length} of $\varrho(\gamma)$ is 
	\begin{align}\label{eq:translation}
		L_{\varrho(\gamma)}:=\inf_{P\in \cC}d(P,\varrho(\gamma)P).
	\end{align}
	If  
	there exists $P_0 \in \cC$ such that $$\inf_{P \in \cC} d(P, gP)=d(P_0, gP_0),$$ then  $\varrho(\gamma)$ is called a \emph{semisimple isometry}. For notational simplicity, we write $L_\gamma$ instead of	$L_{\varrho(\gamma)}$ if no confusion arises. 
\end{dfn} 
The  definition of logarithmic energy growth of a harmonic map was   introduced in \cite{DMrs,DMks}. A slightly more intrinsic version is provided in \cite{BDDM}, which we recall here. 
\begin{dfn}[logarithmic energy growth]\label{def:log energy}
	Let $X$ be a smooth quasi-projective variety, $G$ be a reductive  algebraic group over a non-archimedean local field $K$,  and  let   $\varrho: \pi_1(X) \rightarrow G(K)$  be a Zariski dense representation.  A $\varrho$-equivariant  harmonic map $ \tilde{u}: \widetilde{X} \rightarrow \Delta(G)$ has \emph{logarithmic energy growth} if for any holomorphic map $f:\bD^*\to X$ with no essential singularity at the origin (i.e. for some,   thus any, smooth projective compactification $\overline{X}$ of $X$, $f$ extends to a holomorphic map $\bar{f}:\bD\to \overline{X}$),  there is a positive constant $C$ such that  for any $r\in (0, \frac{1}{2})$,  one has
	\begin{equation} \label{ondisks3}
		- \frac{L^2_\gamma}{2\pi} \log {r} \leq E^{ u_f}[\mathbb D_{r,\frac{1}{2}}] \leq  -\frac{L^2_\gamma}{2\pi} \log r+C,
	\end{equation} 
	where $L_\gamma$ is the translation   length of $\varrho(\gamma)$ with $\gamma \in \pi_1(X)$ corresponding to the loop $\theta\mapsto f(\frac{1}{2}e^{i\theta})$.  
\end{dfn}

\subsection{A Bertini-type theorem}
\begin{proposition}[\protecting{\cite[Proposition 2.11]{BDDM}}]\label{prop:Bertini}
	Let \((\overline{X},\Sigma)\) be a log smooth pair with \(n:=\dim X\geqslant 2\).  Fix a projective embedding \(\iota:\overline{X}\hookrightarrow \bP^N\) and denote by $L:=\iota^*\sO_{\bP^N}(3)$.  For any element $s \in H^0(\overline{X},L)$, we  set 
	$\overline{Y}_{\!\! s} := s^{-1}(0)$, $Y_{\! s}: =\overline{Y}_{\!\! s} \backslash \Sigma$,  and denote by $\iota_{Y_{\! s}}:Y_{\! s} \rightarrow X$  the inclusion map. Let \begin{align}\label{eq:U}
		{\mathbb U}= \{ s \in  H^0(\overline{X},L)\mid  \overline{Y}_{\!\! s} \  \mbox{ is smooth} \ \mbox{and} \ \overline{Y}_{\!\! s} + \Sigma \  \mbox{ is a  normal crossing divisor} \}. 
	\end{align}  
	For $q \in X$, consider  the subspace
	\begin{align}\label{eq:Uq}
		V(q) =\{s \in H^0(\overline{X}, L)\mid s(q)=0\} 
		\mbox{ and } 
		{\mathbb{U}}(q) = \bU \cap V(q).
	\end{align}
	Then 
	\begin{thmlist} 
		\item The set \(\bU(q)\) is non-empty.
		\item\label{item:tangent}   For any $p,q\in X$, and $v\in T_{p}X$, there exists some  \(s\in \bU(q)\) such that $p\in Y_{\! s}$ and $Y_{\! s}$ is tangent to $v$.
		\item \label{item:Lefschetz}For each $s\in \bU$, $\pi_1(Y_s)\to \pi_1(X)$ is surjective.   
	\end{thmlist}
\end{proposition} 
Note that the last assertion follows from the Lefschetz theorem in  \cite{Eyr04}.
\section{Pluriharmonic maps to Euclidean buildings}
In this section we prove \cref{main:harmonic}.  
As a warm-up, we begin by considering the following special case.
\begin{lem}\label{lem:special}
	Let $\varrho:\pi_1(\bC^*)\to (\bR,+)$ be a representation. Consider $\exp:\bC\to \bC^*$ as the universal covering map. Then there exists a $\varrho$-equivariant pluriharmonic map  $\tilde{u}:\bC\to \bR$ with logarithmic energy growth.  Furthermore,
	\begin{thmlist}
		\item the holomorphic 1-form $\partial \tilde{u}=\exp^*(\zeta d\log z)$ for some $\zeta \in \sqrt{-1}\bR$. 
		\item such $\tilde{u}$ is unique up to a translation by a constant.  
	\end{thmlist} 
\end{lem}
\begin{proof}
	Let $\gamma$ be the equivalent class in $\pi_1(\bC^*)$ represented  the loop $\theta\mapsto  e^{\sqrt{-1}\theta}$ in $\bC^*$. Then $\varrho(\gamma)(x)=x+a$ for some $a\in \bR$.  	Define a map
	\begin{align*}
		\tilde{u}:  \bC &\rightarrow \mathbb R\\
		w &\mapsto \frac{1}{2}\int_{0}^w (  \exp^*(-\sqrt{-1}\frac{a}{2\pi}d \log z+\sqrt{-1} \frac{a}{2\pi}d\log \bar z)).
	\end{align*} 
	Then $\tilde{u}(w)=\frac{a}{2\pi} {\rm Im}(w)$. Thus, $\tilde{u}(w+2\pi \sqrt{-1})=\tilde{u}(w)+a$, that     is a $\varrho$-equivariant. We have
	moreover
	$\partial \tilde{u}(w)=  -\sqrt{-1}\frac{a}{4\pi}d   w$, which is a holomorphic 1-form on $\bC^*$. It follows that
	$
	\partial\db \tilde{u}\equiv 0.
	$  Thus, $\tilde{u}$ is pluriharmonic, and     $$\partial \tilde{u}=\exp^*(-\sqrt{-1}\frac{a}{4\pi}d \log z).$$ 
	This proves Item (i).

	Endow $\bD^*$ with the standard Euclidean metric $\sqrt{-1}\frac{dz\wedge d\bar{z}}{2}$. However, note that the energy is independent of the choice of metric on the Riemann surface.  We can easily compute the energy of $u$ in the annulus $\bD_{r,1}:=\{r <|z| < 1\} \subset \mathbb C^*$:  
	\begin{align*}
		E^{ {u}}[\bD_{r,1}] &= \int_{\bD_{r,1}} |d {u}|^2\frac{\sqrt{-1}dz\wedge d\bar{z}}{2}  \\
		&=\int_{\bD_{r,1}} |\frac{a}{2\pi}d\theta|^2rdr\wedge d\theta\\
		& =(\frac{a}{2\pi})^2\int_{0}^{2\pi} d\theta \int_{r}^{1} d\log r=\frac{a^2}{2\pi}\log\frac{1}{r}. 
	\end{align*} 
	By \cref{def:translation},   the translation length $L_\gamma=|a|$.   By \cref{def:log energy}, $\tilde{u}$ has logarithmic energy growth.   In conclusion, $\tilde{u}$ is a pluriharmonic map with logarithmic energy growth.
	
	\medspace

	Let us prove Item (ii). If $\tilde{v}:\bC\to \bR$ is another $\varrho$-equivariant pluriharmonic map with logarithmic energy growth, then $\partial \tilde{v}$ is a holomorphic 1-form, which descends to 1-form $\eta$ on $\bC^*$ such that  $\exp^*\eta=\partial \tilde{v}$.  By \cite{BDDM}, $\eta$ is a logarithmic form on $\bC^*$. Hence there exists a constant $b=b_1+\sqrt{-1}b_2$ with $b_i\in \bR$ such that   $\eta=bd\log z$.   Note that $d \tilde{v}=\exp^*(\eta+\bar{\eta})$.   It follows that 
	\begin{align}\label{eq:energy2}
		a=\tilde{v}(w+2\pi \sqrt{-1})-\tilde{v}(w)=\int_\gamma (\eta+\bar{\eta})=-4\pi b_2.
	\end{align}
	Hence $b_2=-\frac{a}{4\pi}$.
	
	Let us compute the energy of $\tilde{v}$ on the annulus $\bD_{r,1}$. We    have 
	\begin{align}\label{eq:energy}
		E^{ {v}}[\bD_{r,1}] &= \int_{\bD_{r,1}} |dv|^2\frac{\sqrt{-1}dz\wedge d\bar{z}}{2}  \\\nonumber
		&=\int_{\bD_{r,1}} |2b_1d\log r-2b_2d\theta|^2rdr\wedge d\theta\\\nonumber
		& =\big((\frac{a}{2\pi})^2+4b_1^2\big)\int_{0}^{2\pi} d\theta \int_{r}^{1} d\log r\\\nonumber
		&=\frac{a^2}{2\pi}\log\frac{1}{r}+ 8\pi b_1^2d\log\frac{1}{r}. 
	\end{align}  
	By \cref{ondisks3},  $b_1=0$.  This implies that $\partial \tilde{u}=\partial \tilde{v}$. Hence $d(\tilde{u}-\tilde{v})=0$. Therefore, $\tilde{u}$ is unique up to a translation.  The lemma is proved. 
\end{proof}
Let $(\overline{X},\Sigma)$ be a log smooth pair.  
Let us recall the definition of residue of a logarithmic form $\eta\in H^0(\overline{X},\Omega_{\overline{X}}(\log \Sigma))$ around an irreducible component $\Sigma_i$ of $\Sigma$. We fix an admissible coordinate $(U;z_1\ldots,z_n)$ centered at some point $x_0\in \Sigma_i$ away from the crossings of $\Sigma$ such that $(z_1=0)=U\cap \Sigma_i=U\cap \Sigma$.  Then we can write $\eta=h_1(z)d\log z_1+\sum_{i=2}^{n}h_i(z)dz_i$. We define
\begin{align}\label{eq:residue}
	{\rm Res}_{\Sigma_i} \eta:=h_1(0).
\end{align} 
Note that such definition does not depend on the choice of local coordinate system. 
\begin{dfn}[Pure imaginary logarithmic form]\label{def:pure}
	Let $(\overline{X},\Sigma)$ be a log smooth pair. A logarithmic form $\eta$ is \emph{pure imaginary} if for each irreducible component $\Sigma_i$ of $\Sigma$, the residue of $\eta$ at $\Sigma_i$ is a pure imaginary number. 
\end{dfn}
Note that \cref{def:pure} does not depend on the choice of compactification of $X=\overline{X}\backslash\Sigma$.

\begin{proposition}\label{lem:char}
	Let $(\overline{X},\Sigma)$ be a log smooth pair. Let $\varrho:\pi_1(X)\to (\bR,+)$ be a representation.  If there exists a $\varrho$-equivariant pluriharmonic map $\tilde{u}:\widetilde{X}\to \bR$, then $\tilde{u}$ has logarithmic energy growth if and only if $\partial \tilde{u}$ descends to a logarithmic form $\eta\in H^0(\overline{X},\Omega_{ \overline{X}}(\log \Sigma))$, that is pure imaginary.  
\end{proposition}
\begin{proof}
	We write  $\Sigma=\sum_{i=1}^{m}\Sigma_i$ into a sum of irreducible components.  Fix some $i\in \{1,\ldots,m\}$.  Choose a   point $x_0\in \Sigma_i\backslash\cup_{j\neq i}\Sigma_j$. We take a small embedded disk $f:\bD\to \overline{X}$ such that $f^{-1}(\Sigma)=f^{-1}(\Sigma_i)=\{0\}$ and $f$ is transverse to $\Sigma_i$ at $x_0$.   Let $\gamma\in \pi_1(X)$ be the element  representing the loop $\theta\mapsto f(\frac{1}{2}e^{i\theta})$. Let $\bH$ be the left half plane of $\bC$. Then  $\exp:\bH\to \bD^*$ is the universal covering map. Let $\tilde{f}:\bH\to \widetilde{X}$ be the lift of $f$ between universal covers. Then $\tilde{u}\circ\tilde{f}:\bH\to \bR$ is $f^*\varrho$-equivariant pluriharmonic map and let $u_f$ be the section defined in \cref{sec:pullback}. 
	
	If $\tilde{u}$ has logarithmic energy growth, then by \cite{BDDM}, $\partial \tilde{u}$ descends to a logarithmic form $\eta\in H^0(\overline{X},\Omega_{ \overline{X}}(\log \Sigma))$. Let us prove that $\eta$ is pure imaginary.  By \cref{def:translation}, the translation length $L_\gamma$ is given by
	$$
	L_\gamma=\left|\int_{\gamma} (f^*\eta+f^*\bar{\eta})\right|=\left| 2\pi \sn({\rm Res}_{\Sigma_i}\eta-\overline{{\rm Res}_{\Sigma_i}\eta})\right|.
	$$
	Since $\eta$ has logarithmic poles, there is some $h(z)\in \sO(\bD)$ such that $f^*\eta=h(z)d\log z$. Write $h(z)=h_1(z)+\sqrt{-1}h_2(z)$, where $h_i(z)$ are real harmonic functions on $\bD$. Then 
	\begin{align}\label{eq:translate}
		L_\gamma=|4\pi h_2(0)|.
	\end{align} 
	The energy 
	\begin{align}\label{eq:local}
		E^{u_f}[\bD_{r,1}] &= \int_{\bD_{r,1}} |f^*\eta+f^*\bar{\eta}|^2\frac{\sqrt{-1}dz\wedge d\bar{z}}{2} \\ \nonumber
		&= \int_{\bD_{r,1}} |h(z)d\log z+ \overline{h(z)}d\log \bar{z}|^2\frac{\sqrt{-1}dz\wedge d\bar{z}}{2} \\\nonumber
		&	= \int_{\bD_{r,1}} |2h_1(z)d\log t- 2h_2(z) d\theta|^2tdt\wedge d\theta \\\nonumber
		&= \int_{r}^{1} \int_{0}^{2\pi}|2h_1(te^{\sn\theta})|^2d\log t\wedge d\theta\\\nonumber
		& +  \int_{r}^{1} \int_{0}^{2\pi}|2h_2(te^{\sn\theta})|^2 d\log t \wedge  d\theta 
	\end{align}  
	Since $|h_i(z)|^2$ are subharmonic functions on $\bD$,  by the mean value inequality there exists a constant $C>0$ such that
	\begin{align}\label{eq:computation}
		8\pi(|h_1(0)|^2+|h_2(0)|^2)\log\frac{ 1}{r}\leq 	E^{u_f}[\bD_{r,1}] \leq  	8\pi(|h_1(0)|^2+|h_2(0)|^2)\log\frac{ 1}{r}+C,\quad \forall\ \  r\in (0,1).
	\end{align}
	By \cref{def:log energy}, we have $h_1(0)=0$. Hence $\eta$ is pure imaginary. 
	
	We now assume that $\eta$ is pure imaginary. Let $g:\bD\to \overline{X}$ be any holomorphic map such that $g^{-1}(\Sigma)=\{0\}$.  Then $g^*\eta=h(z)\log z$ with  $h(0)\in \sn\bR$. We denote by $u_g$ the section of $\bD^*\times_{g^*\varrho}\bR\to \bD^*$ defined in \cref{sec:pullback}.    By the same manner as \eqref{eq:translate} and \eqref{eq:computation}, we can show that $u_{g}$ has logarithmic energy growth. By \cref{def:log energy}, $u$ has logarithmic energy growth. 
\end{proof}

We can extend \cref{lem:special} to the case of semi-abelian varieties. 
\begin{proposition}\label{prop:abelian}
	Let $A$ be a semiabelian variety and let $\varrho:\pi_1(A)\to(\bR^N,+)$ be a representation. Then there is a $\varrho$-equivariant pluriharmonic map $u:\widetilde{A}\to \bR^N$ with logarithmic energy growth.     Such pluriharmonic map is unique up to translation.
\end{proposition}

\begin{proof}
	Note that there is a short exact sequence
	$$
	0\to (\bC^*)^k\stackrel{j}{\to} A\stackrel{\pi}{\to} A_0\to 0,
	$$
	where $A_0$ is an abelian variety.
	Let $\overline{A}$ be the canonical compactification of $A$ such that $\pi:A\to A_0$ extends to a $(\bP^1)^k$-fiber bundle $$
	0\to (\bP^1)^k\stackrel{\bar{j}}{\to} \overline{A}\stackrel{\bar{\pi}}{\to} A_0\to 0.
	$$   Let $\Sigma:=\overline{A}\setminus A$ which is a smooth divisor.  
	Let $V\subset 	H^0(\overline{A}, \Omega_{\overline{A}}(\log \Sigma))$ be the $\bR$-linear subspace consisting of logarithmic forms, whose resides at each irreducible component of $\Sigma$ are pure imaginary.  Let $d:=\dim A_0$.   
	\begin{claim}\label{claim:iso}
		We have	$\dim_{\bR}V=2d+k$. The $\bR$-linear map 
		\begin{align}\label{eq:isopure}
			\Psi:	V &\to H^1(A,\bR)\\\nonumber
			\eta&\mapsto  \{\frac{\eta +\bar{\eta}}{2} \}
		\end{align} 
		is an isomorphism of $\bR$-vector spaces. 
	\end{claim}
	\begin{proof}
		Note that $\dim_{\bC} 	H^0(\overline{A}, \Omega_{\overline{A}}(\log \Sigma))=d+k$ and $\dim_{\bR}H^1(A,\bR)=2d+k$.     
		We choose a $\bC$-basis $\eta_1,\ldots,\eta_d;\xi_1,\ldots,\xi_k$ for $H^0(\overline{A}, \Omega_{\overline{A}}(\log \Sigma))=d+k$ such that $\{\eta_1,\ldots,\eta_d\}\subset \pi^*H^0(A_0,\Omega_{A_0})$. The residues of $\eta_i$ at each component of $\Sigma$ is thus zero. Let $(w_1,\ldots,w_k)$ be the canonical coordinate of $(\bC^*)^k$. Then $j^*\xi_m=\sum_{i=1}^{k}a_{mi}d\log w_i$ with $(a_{m1},\ldots,a_{mk})\in \bC^k$.  Note that $\bar{j}^*\xi_1,\ldots,\bar{j}^*\xi_m$ is a $\bC$-basis of $H^0((\bP^1)^k,\Omega_{(\bP^1)^k}(\log D))$, where $D:=(\bP^1)^k\setminus (\bC^*)^k$. Note that $d\log w_1,\ldots,d\log w_k$ is  also   $\bC$-basis of $H^0((\bP^1)^k,\Omega_{(\bP^1)^k}(\log D))$. We can thus replace  $\xi_1,\ldots, \xi_m$ by some $\bC$-linear combination such that $\bar{j}^*\xi_i=\sqrt{-1}d\log w_i$ for each $i=1,\ldots,k$. This implies that each $\xi_i$ has pure imaginary residues at each   irreducible component of $\Sigma$. Then we have  $$V:={\rm Span}_{\bR}\{\xi_1,\ldots,\xi_k,\eta_1,\ldots,\eta_d,i\eta_1,\ldots,i\eta_d\}.$$
		We can see that $\Psi $ is a $\bR$-isomorphism. 
	\end{proof}
	Let the homomorphism ${\rm pr}_i:(\bR^N, +)\to (\bR, +)$ be  the projection into $i$-th factor. Then ${\pr}_i\circ\varrho:\pi_1(A)\to (\bR,+)$ is a representation which can be identified with an element $\lambda_i\in H^1(A,\bR)$ as  $H^1(A,\bR)\simeq {\rm Hom}(H_1(A,\bZ), \bR)$. Denote by $\zeta_i:=\Psi^{-1}(\lambda_i)$. We define   
	\begin{align*}
		\tilde{u}_i: \widetilde{A}&\to \bR\\
		z&\mapsto \frac{1}{2}\int_{0}^{z}\pi_A^*(\zeta_i+\bar{\zeta}_i).
	\end{align*} 
	Then we obtain a smooth map $\tilde{u}:\tilde{A}\to \bR^N$ defined by $\tilde{u}=(\tilde{u}_1,\ldots,\tilde{u}_N)$. This map is pluriharmonic as $ \bar{\d}\d \tilde{u}=(\frac{1}{2}\bar{\d}\pi_A^*\zeta_1,\ldots,\frac{1}{2}\bar{\d}\pi_A^*\zeta_N)=(0,\ldots,0)$.   One can verify that $\tilde{u}$ is $\varrho$-equivariant.  Indeed, for any $x\in \widetilde{A}$ and any $\gamma\in \pi_1(A)$, we have
	\begin{align}\label{eq:definetran}
		\tilde{u}_i(\gamma.x)-\tilde{u}_i(x)=\int_{\gamma}\frac{1}{2}  (\zeta_i+\bar{\zeta}_i)=\lambda_i(\gamma)={\rm pr}_i\circ\varrho(\gamma)(\tilde{u}_i(x))-\tilde{u}_i(x). 
	\end{align}
	Let us  prove that $\tilde{u}$ has logarithmic energy growth.  Since $\partial \tilde{u_i}=\frac{1}{2}\pi_X^*\zeta_i$, where $\zeta_i$ is a pure imaginary logarithmic 1-form,   by \cref{lem:char},   $\tilde{u}_i:\widetilde{A}\to \bR$ is a ${\rm pr}_i\circ\varrho$-pluriharmonic map  with logarithmic energy growth.    Let $f:\bD\to \overline{A}$ be any holomorphic map such that $f^{-1}(\Sigma)=\{0\}$. Let $\gamma$ be the element in $\pi_1(X)$ represented by the loop $\theta\mapsto f(\frac{1}{2}e^{\sn\theta})$.  Let $L_i$ be the translation length of ${\rm pr}_i\circ \varrho(\gamma)$.    It follows that there exists a constant $C>0$ such that for each $i\in \{1,\ldots,N\}$, we have
	$$
	\frac{L_i^2}{2\pi}\log \frac{1}{r}\leq E^{(u_i)_f}[\bD_{r,1}]\leq \frac{L_i^2}{2\pi}\log \frac{1}{r}+C,\quad \forall \ \ r\in (0,1).
	$$
	Note that $$ E^{u_f}[\bD_{r,1}]=\sum_{i=1}^{N}E^{(u_i)_f}[\bD_{r,1}],  \quad \mbox{and}\quad L^2_{\varrho(\gamma)}=\sum_{i=1}^{N}L_i^2.$$ We thus have
	$$
	\frac{L_{\varrho(\gamma)}^2}{2\pi}\log \frac{1}{r}\leq E^{u_f}[\bD_{r,1}]\leq \frac{L_{\varrho(\gamma)}^2}{2\pi}\log \frac{1}{r}+C,\quad \forall \ \ r\in (0,1).
	$$
	Thus, $\tilde{u}$ is a pluriharmonic map with logarithmic energy growth. 
	
	Let us prove the uniqueness assertion. Let $\tilde{v}=(\tilde{v}_1,\ldots,\tilde{v}_N):\widetilde{A}\to \bR^N$ be another $\varrho$-equivariant pluriharmonic map with logarithmic energy growth. 
	Then for each  $i\in \{1,\ldots,N\}$, $\tilde{v}_i:\widetilde{A}\to \bR$ is a ${\rm pr}_i\circ\varrho$-pluriharmonic map  with logarithmic energy growth.     By \cref{lem:char}, $\partial \tilde{v}_i$ descends to a logarithmic form $\frac{1}{2}\omega_i$ that is pure imaginary. By \eqref{eq:definetran}, for any $\gamma\in \pi_1(X)$,  we have
	\begin{align*}
		\tilde{v}_i(\gamma.x)-\tilde{v}_i(x)=\int_{\gamma}\frac{1}{2}  (\omega_i+\bar{\omega}_i)={\rm pr}_i\circ\varrho(\gamma)(\tilde{v}_i(x))-\tilde{v}_i(x)\\
		={\rm pr}_i\circ\varrho(\gamma)(\tilde{u}_i(x))-\tilde{u}_i(x)\lambda_i(\gamma)=\int_{\gamma}\frac{1}{2}  (\zeta_i+\bar{\zeta}_i).
	\end{align*}  
	By \cref{claim:iso}, we have $\zeta_i=\omega_i$. It follows that $d \tilde{u}=d\tilde{v}$. Hence $\tilde{u}-\tilde{v}$ is a constant.  
	The proposition is proved.  
\end{proof}

Let us prove \cref{main:harmonic}, except for \cref{item:singular}, whose proof is deferred to \cref{sec:singular}.
\begin{proof}[Proof of \cref{main:harmonic}]
	Consider the enlarged Bruhat-Tits building $\Delta(G)$. It is indeed the product of the Bruhat-Tits building of    $\Delta(\sD G)$ where $\sD G$ is the derived group of $G$, with a real Euclidean space $V:=\bR^N$ such that $G(K)$ acts on $V$ by translation (cf. \cite{KP23}).  The fixator of any point in $\Delta(G)$ is an open and bounded subgroup of $G(K)$.   Note that there is a natural action of $\sD G(K)$    on $\Delta(\sD G)$. The action of $G(K)$ on $\Delta(\sD G)$ is given by the composition of $G(K)\to \sD G(K)$ with the action of    of $\sD G(K)$    on $\Delta(\sD G)$. 
	
	We consider the representation $\sigma:\pi_1(X)\to \sD G(K)$  induced by $\varrho$, which is Zariski dense. By \cite{BDDM}, there exists a $\sigma$-equivariant pluriharmonic map $\tilde{u}_0:\widetilde{X}\to \Delta(\sD G)$   with logarithmic energy growth.  
	
	On the other hand, for the action of $G(K)$ on $V$, it induces a representation 
	$
	\tau:\pi_1(X)\to (V,+)
	$.  	Let $a:X\to A$ be the quasi-Albanese map, and $\tilde{a}:\widetilde{X}\to \widetilde{A}$ be a lift of $a$ between  universal covers.     
	Note that $\tau$ factors through a representation $\tau':\pi_1(A)\to  (V,+)$.  By \cref{prop:abelian}, there exists a $\tau'$-equivariant pluriharmonic map $\tilde{v}:\widetilde{A}\to V$   which has logarithmic energy growth.  Therefore,  $\tilde{v}\circ \tilde{a}:\widetilde{X}\to V$ is a $\tau$-equivariant pluriharmonic map.  Since  $\partial \tilde{v}$ descends to a tuple of logarithmic 1-forms $\{\omega_1,\ldots,\omega_m \} $ on $A$ that are pure imaginary,   it implies that $\partial \tilde{v}\circ \tilde{a}$ descends to  $\{a^*\omega_1,\ldots,a^*\omega_m \}$, that are  also pure imaginary logarithmic 1-forms on $(\overline{X},\Sigma)$.  By \cref{lem:char}, $\tilde{v}\circ\tilde{a}$ has logarithmic energy growth.   We define
	\begin{align}\label{eq:harmonic}
		\tilde{u}:	\widetilde{X}&\to \Delta(\sD G)\times V\\\nonumber
		x&\mapsto  (\tilde{u}_0(x), \tilde{v}\circ \tilde{a}(x)).
	\end{align} 
	Since $\varrho=(\sigma,\tau)$,  $\tilde{u}$ is $\varrho$-equivariant pluriharmonic map. Since both $\tilde{u}_0$ and $\tilde{v}\circ\tilde{a}$ have logarithmic energy growth,   $\tilde{u}$ also  has logarithmic energy growth.   The existence assertion in \cref{item:existence} is established.
	
	\medspace
	
	Let us prove \cref{item:harmonic}. By \cite[Theorem A]{BDDM}, $\tilde{u}_0$ is harmonic with respect to an arbitrary K\"ahler metric $\omega$ on $\widetilde{X}$.  The pluriharmonicity of $\tilde{v}\circ\tilde{a}$  yields that $\hess \tilde{v}\circ\tilde{a}\equiv0$. Thus, 
	$$
	\Delta \tilde{v}\circ\tilde{a}=-2\sn \Lambda_\omega \hess \tilde{v}\circ\tilde{a}\equiv0,
	$$
	where $\Lambda_\omega $ denotes the contraction with $\omega$.  It follows that $\tilde{v}\circ\tilde{a}$ is harmonic with respect to the metric $\omega$.  Therefore, $\tilde{u}$  is harmonic with respect to the metric $\omega$. 
	
	\medspace
	
	Finally, we prove \cref{item:pullback}. Let $\overline{Y}$ be a smooth projective compactification with $\Sigma_Y:=\overline{Y}\backslash Y$ a simple normal crossing divisor such that $f$ extends to morphism $\bar{f}:\overline{Y}\to \overline{X}$.   Then by \cite[Theorem A]{BDDM}, $\tilde{u}_0\circ \tilde{f}:\widetilde{Y}\to \Delta(G)$ is a pluriharmonic map with logarithmic energy growth.  By the above arguments, $\partial \tilde{\nu}\circ\tilde{a}\circ\tilde{f}:\widetilde{Y}\to V$ descends to logarithmic forms  $\{(a\circ f)^*\omega_1,\ldots,(a\circ f)^*\omega_m \}$  on the log smooth pair $(\overline{Y},\Sigma_Y)$, that are pure imaginary. By \cref{lem:char},  $\tilde{\nu}\circ\tilde{a}\circ\tilde{f}$  is pluriharmonic with  logarithmic energy growth. Thus, $\tilde{u}\circ\tilde{f}$ is pluriharmonic  with  logarithmic energy growth.  
	The theorem is proved. 
\end{proof}

\section{Multivalued section and  spectral cover} \label{sec:KZthm}
The notion of \emph{multivalued sections} of a holomorphic vector bundle over a complex manifold has appeared in \cite{CDY22, DW24}, and has proven to be important in studying the geometry of complex algebraic varieties that admit a local system over a non-archimedean local field. In this section, we provide a more systematic description of multivalued sections and their properties in a general setting. The construction of multivalued logarithmic 1-forms on log smooth pairs here is equivalent to, though simpler than,   that in \cite{CDY22}. 
\subsection{Construction of spectral cover} \label{subsec:spectral}
We start with the following definition. 
\begin{dfn}[Multivalued section]\label{def:multivalued}
	Let $X$ be a complex manifold, and let $E$ be a holomorphic vector bundle on $X$. A   {multivalued} (holomorphic) section of $E$ on $X$, denoted by $\eta$,   is a formal sum $Z_\eta=\sum_{i=1}^{m}n_iZ_i$ where $n_i\in \bN^*$,  and each  $Z_i$ is an irreducible closed subvariety of $E$, such that the natural map $Z_i\to X$ is a finite and surjective. 
	
	A multivalued section $\eta$ is \emph{splitting}, if for  each point $x\in X $, it has an open neighborhood $\Omega_x$ and holomorphic sections  $\{\omega_1,\ldots,\omega_m\}\subset \Gamma(\Omega_x,E|_{\Omega_x}),$  such that $Z_\eta|_{\Omega_x}$ is the graph of $\{\omega_1,\ldots,\omega_m\}$. 
\end{dfn} 
Note that in \cite{CDY22}, multivalued sections are splitting ones. 

Let $X$ be a complex manifold, and let $E$ be a holomorphic vector bundle on $X$. Assume that   $\eta$ is a splitting multivalued section of $E$.  Let  $T$ be a   formal variable.  Consider   $\prod_{i=1}^{m}(T-\omega_i)=:T^m+\sigma_1 T^{m-1}+\cdots+ \sigma_m$, where $\{\omega_1,\ldots,\omega_m\}$ are local sections of $E$ whose graph is $Z_\eta$.  Then $\sigma_i$ is a local section of ${\rm Sym}^iE$.  One can see that $\sigma_i$  is a global section in $H^0(X, {\rm Sym}^iE)$. We call $T^m+\sigma_1 T^{m-1}+\cdots+ \sigma_m$ the \emph{characteristic polynomial} of $\eta$, and denote it by $P_\eta(T)$. 
\begin{proposition}\label{prop:spectral}
	Let $X$ be a smooth projective  variety endowed with a holomorphic vector bundle $E$.    
	Let $X'\subset X$ be a topological dense open set. Let $\eta$ be  a splitting  \emph{multivalued}  section of $E|_{X'}$ over $X'$.  Assume that for the characteristic polynomial $P_\eta(T)=T^m+\sigma_1 T^{m-1}+\cdots+ \sigma_m$ of $\eta$, its coefficient  $\sigma_i\in H^0(X', {\rm Sym}^iE|_{X'})$  extends to a section in $H^0(X, {\rm Sym}^iE)$ for each $i$. Then $\eta$ extends to a multivalued section of $E$. 
	
	Furthermore, there exists  a finite, (possibly ramified) Galois cover $\pi:\xsp\to X$ with  Galois group $G$ from a projective normal variety such that $\pi^*\eta$ becomes single-valued, i.e., there exists sections $\{\eta_1,\ldots,\eta_m\}\subset H^0(\xsp,\pi^*E)$ such that $\pi^*\eta=\{\eta_1,\ldots,\eta_m\}$.  The group $G$ acts on  $\{\eta_1,\ldots,\eta_m\}$ as a permutation. 
\end{proposition}
\begin{dfn}[Spectral cover]
	The above Galois cover $\pi$ is called the \emph{spectral cover} of $X$ with respect to $\eta$.  
\end{dfn}

\begin{proof}
	Denote by $\mu:E\to X$ be projection map.	Let $\lambda\in H^0(E, \mu^*E)$ be the Liouville section defined by $\lambda(e)=e$ for any $e\in E$. Consider the section
	$$
	P_\eta(\lambda):= \lambda^m+\mu^*\sigma_1 \lambda^{m-1}+\cdots+ \mu^*\sigma_m \in H^0(E, \mu^*\Sym^mE).
	$$  
	Let $Z\subset E$ be  the zero locus  of $P_\eta(\lambda)$ (here we count multiplicities).  By assumption, one can see that, $Z|_{X'}=Z_\eta$.   Moreover, $\mu|_{Z}:Z\to X$ is a finite morphism.  To show that $Z$ is a multivalued section of $E$, we need to prove that, for each irreducible component $Z_i$ of $Z$, $\mu|_{Z_i}:Z_i\to X$ is surjective. 
	
	Let $Z^{\rm norm}$ be the normalization of $Z$ which might not be connected. Then  the natural morphism $q: Z^{\rm norm}\to X$ is finite. Consider the locus $X^\circ$ of $X$ such that  $q$  is \'etale. Then $X^\circ$ is a Zariski dense open set of $X$.  One can see that $X^\circ\supset X'$.   Set $Z^{\rm norm}_\circ:=q^{-1}(X^\circ)$  and $Z^\circ:=(\mu|_Z)^{-1}(X^\circ)$. Note that  $Z^{\rm norm}_\circ$ is the normalization of $Z^\circ$.
	\begin{claim}\label{claim:extended section}
		$\eta$ extends to a splitting multivalued section of $E$ on $X^\circ$.  
	\end{claim}
	\begin{proof} 
		Note that $q: Z^{\rm norm}_\circ\to X^\circ$ is \'etale of degree $m$. 	Hence for every $x\in X^\circ$, it has neighborhood $\Omega_x$ such that  $q^{-1}(\Omega_x)$ is isomorphic to $m$ copy of $\Omega_x$. Thus it gives rise to $m$ natural local sections $s_1,\ldots,s_m:\Omega_x\to Z^{\rm norm}_\circ$  of  $q: Z^{\rm norm}\to X$ such that $s_1(\Omega_x),\ldots,s_m(\Omega_x)$ correspond  to $m$ components of $q^{-1}(\Omega_x)$. Let $\nu_Z:Z^{\rm norm}\to Z$ be the normalization map. Then $\{\nu_Z\circ s_i: \Omega_x\to Z\subset E\}_{i=1,\ldots,m}\subset H^0(\Omega_x, E|_{\Omega_x})$.  Note that the graph of $ \{\nu_Z\circ s_1,\ldots,\nu_Z\circ s_m\}$ is $Z|_{\Omega_x}$.     The claim is proved. 
	\end{proof}
	We still denote by   $\eta$    the extended multivalued section of $E|_{X^\circ}$.  
	\begin{claim}\label{claim:monodromy}
		The \'etale morphism $q|_{Z_\circ^{\rm norm}}:Z^{\rm norm}_\circ\to X^\circ$ gives rise to  a  representation $\phi:\pi_1(X^\circ)\to\mathfrak{S}_m$ where $\mathfrak{S}_m$ is  the symmetric group of $m$ elements. Let $\pi:Y^\circ\to X^\circ$   be the Galois \'etale cover  corresponding to the finite index subgroup $\ker \phi$ of $\pi_1(X^\circ)$. Then 
		\begin{itemize}
			\item  the normalization of the base change  $Z^\circ\times_{X^\circ} Y^\circ$ is a quasi-projective   variety with $m$ connected component such that each component is isomorphic to $Y^\circ$ under the natural map $(Z^\circ\times_{X^\circ} Y^\circ)^{\rm norm}\to Y^\circ$.  
			\item There are sections  $\{\eta_1,\ldots,\eta_m\}\subset H^0(Y^\circ, \pi^*E)$ such that $\{\eta_1,\ldots,\eta_m\} =\pi^*\eta$.
			\item  $G$ acts on $\{\eta_1,\ldots,\eta_m\}$ as a permutation. 
		\end{itemize} 
	\end{claim}
	\begin{proof}
		We fix a base point $x_0\in X'$. There exists an open   neighborhood $\Omega_{x_0}$ of $x_0$ such that,  the multivalued section $\eta|_{\Omega_{x_0}}$ is given  by  sections $\{s_1,\ldots,s_m\}\subset H^0(\Omega_{x_0},E|_{\Omega_{x_0}})$.   Consider any loop $\gamma$ of $X^\circ$ based at $x_0$. Since $q|_{Z_\circ^{\rm norm}}:Z^{\rm norm}_\circ\to X^\circ$ is \'etale, by \cref{def:multivalued}, the  transport of  $\{s_1,\ldots,s_m\}$    along $Z^{\rm norm}_\circ|_{\gamma}$ gives a  permutation of $\{s_1,\ldots,s_m\}$ hence an element in the symmetric group $\mathfrak{S}_m$ of $m$ elements. We can see that it only depends  on the choice of homotopy class of $\gamma$ and thus it corresponds to a representation $\phi:\pi_1(X^\circ)\to\mathfrak{S}_m$. Let $\pi: Y^\circ\to X^\circ$  be the Galois \'etale cover with the Galois group $G:=\pi_1(X^\circ)/\ker \phi$. Then for any loop $\gamma\in \pi_1(Y^\circ)$,  the transport of $\{s_1,\ldots,s_m\}$ along  $Z^{\rm norm}_\circ\times_{X^\circ}Y^\circ|_{\gamma}$ is a trivial permutation, which thus gives rise to holomorphic sections $\{\eta_1,\ldots,\eta_m\}\subset H^0(Y^\circ, \pi^*E)$. It follows that    $\pi^*\eta=\{\eta_1,\ldots,\eta_m\}$.  One can see that $G$ acts on $\{\eta_1,\ldots,\eta_m\}$ as a permutation. 
		
		Let $W^\circ\subset \pi^*E$ be the graph variety of $\{\eta_1,\ldots,\eta_m\}$. One can see that $W^\circ$ coincides with $Z^\circ\times_{X^\circ}Y^\circ$.  Hence the normalization $(Z^\circ\times_{X^\circ}Y^\circ)^{\rm norm}$ is isomorphic to $m$ copy of $Y^\circ$.  The claim is proved. 
	\end{proof}
	Note that $\pi: Y^\circ\to X^\circ$ extends to a ramified Galois cover $Y\to X$ with the Galois group $G$, where $Y$ is a projective normal variety.  We still denote by $\pi:Y\to X$ the extended  cover.  
	Let $\nu:\pi^*E\to Y$ the natural projection map. We have the following commutative diagram:
	\[
	\begin{tikzcd}
		\pi^*E\arrow[r,"f"] \arrow[d,"\nu"]& E\arrow[d,"\mu"]\\
		Y\arrow[r,"\pi"] &X
	\end{tikzcd} 
	\]  Let $\lambda'\in H^0(\pi^*E, \nu^*\pi^*E)$ be the Liouville section.  Consider the section
	$$
	Q(\lambda'):=	\lambda'^{m}+\nu^*\pi^*\sigma_1\lambda'^{m-1}+\cdots+\nu^*\pi^*\sigma_m\in H^0(\pi^*E, \nu^*\pi^*{\rm Sym}^mE).
	$$
	\begin{claim}
		Each $\eta_i$ extends to a holomorphic section in $H^0(Y,\pi^*E)$, which we continue to denote by $\eta_i$.
	\end{claim}
	\begin{proof}
		Let $y_0 \in Y$ and $U$ be an open neighborhood of $y_0$ such that $\pi^*E|_U$ is trivial. Pick any section $e \in \pi^*E^*(U)$. Then $f_i := \eta_i(e)$ is a holomorphic function on $U^\circ := U \cap Y^\circ$. Let $P_d(X_1, \ldots, X_m)$ denote the elementary symmetric polynomial of degree $d$ in $m$ variables, so that
		\[
		\prod_{i=1}^{m}(T - X_i) = T^m + P_1(X_1, \ldots, X_m) T^{m-1} + \cdots + P_m(X_1, \ldots, X_m).
		\]
		We then have 
		\[
		P_d(f_1, \ldots, f_m) = \pi^*\sigma_d(e)|_{U^\circ}.
		\]
		After shrinking $U$, we may assume there exists a constant $M > 0$ such that 
		\[
		2 \max_{d\in \{1,\ldots,m\}} \sup_{y \in U} |\pi^*\sigma_d(e)(y)| ^{\frac{1}{d}} \leq M.
		\] 
		By the classical inequalities between the norms of the roots and coefficients of a polynomial, we deduce that $\max_{d\in \{1,\ldots,m\}}  \sup_{y \in U^\circ} |f_d(y)| \leq M$. Since $Y$ is normal and $U \setminus U^\circ$ is an analytic subset of $U$, it follows that each $f_i$ extends to a holomorphic function on $U$. Hence, $\eta_i$ extends to a holomorphic section in $H^0(Y, \pi^*E)$. This proves the claim.
	\end{proof} 
 Let $W\subset \pi^*E$  be the zero scheme of  $Q'(\lambda')$.  Note that $W|_{\nu^{-1}(Y^\circ)}=W^\circ$, that is   the graph variety of $\{\eta_1,\ldots,\eta_m\}$.  Therefore, over $\nu^{-1}(Y^\circ)$,  we have
	$$
	Q(\lambda')=\prod_{i=1}^{m}(\lambda'-\eta_i). 
	$$
	By continuity, it follows that 
	the above equality holds over the whole $\pi^*E$.  Thus, each irreducible component of $W$ corresponds to the graph of some $\eta_i$, which is therefore mapped surjectively onto $Y$. Since  we have
	$Q(\lambda')=f^*P_\eta(\lambda)$,  $W$ is equal to the scheme theoretic inverse image $f^{-1}(Z)$.  Consequently, each irreducible component of $Z$ is mapped surjectively onto $X$. Hence, $Z$ forms a multivalued section of $E \to X$. We write $\pi: \xsp \to X$ to denote $\pi: Y \to X$. This completes the proof of the proposition. 
\end{proof} 
\subsection{Invariant 1-forms on Bruhat-Tits buildings} \label{subsec:building}
Let $G$ be a reductive algebraic group over a non-archimedean local field. 
Then  $G$ induces   a real Euclidean space $V$ endowed with a Euclidean metric and  an affine Weyl group $W$ acting on $V$ isometrically.  Such group $W$ is a semidirect product   $T\rtimes W^v$,  where $W^v$ is the \emph{vectorial Weyl group}, which is a finite group generated by reflections on $V$, and $T$ is a translation group of $V$.

For any apartment $A$ in $ \Delta(G)$,  there exists an isomorphism $i_A:A\to V$, which is  called a chart. For two charts $i_{A_1}:A_1\to V$ and $i_{A_2}:A_2\to V$, if $A_1\cap A_2\neq\varnothing$,   it satisfies the following properties:
\begin{enumerate}[label=(\alph*)]
	\item $Y:=i_{A_2}(i_{A_1}^{-1}(V))$ is convex.
	\item \label{item:weyl}There is an element $w\in W$ such that $w\circ i_{A_1}|_{A_1\cap A_2}=i_{A_2}|_{A_1\cap A_2}$.
\end{enumerate} 
Let us fix  orthonormal coordinates $(x_1,\ldots,x_N)$ for $V$.  Since $W^v\subset \GL(V)$ acts on $V$ isometrically,  for any $w\in W^v$, $(w^*x_1,\ldots,w^*x_N)$ are orthonormal coordinates  for $V$. We define a subset of $V^*$ by setting
\begin{align}\label{eq:Phi}
	 \Phi:=\{w^*x_i \}_{i\in  \{1,\ldots,N\}; w\in W^v}.
\end{align}
  Since $W^v$ is a finite group, then $\Phi$ is a finite set. 
Note that  $\Phi$ is invariant under the  action by $W^v$.  We   write 
$ \Phi=\{\beta_1,\ldots,\beta_m\}. 
$ 

We define  real affine functions 
\begin{align}\label{eq:beta}
	 \beta_{A,i}:=\beta_i\circ i_{A}(x) 
\end{align}
on $A$ for each $i$.  
\begin{lem}
	If $A_1\cap A_2\neq\varnothing$, then we have $$\{d\beta_{A_1,1},\ldots,d\beta_{A_1,m}\}|_{A_1\cap A_2}=\{d\beta_{A_2,1},\ldots,d\beta_{A_2,m}\}|_{A_1\cap A_2}.$$
\end{lem}
\begin{proof}
	 By \Cref{item:weyl}, there exists an element $w\in W$ such that  
	 $
	 \beta_k\circ i_{A_2}|_{A_1\cap A_2}=  \beta_k\circ w\circ i_{A_1}|_{A_1\cap A_2} 
	 $ 
	 for any $k=1,\ldots,m$. 
	 Recall that $W^v$ permutes $\Phi$. It follows that there exist $a_1,\ldots,a_m\in \bR$ 
	 and  a permutation $\sigma$ of $m$-elements   such that
	 \begin{align}\label{eq:compatible} 
	 	\beta_k\circ i_{A_2}|_{A_1\cap A_2}=  \beta_k\circ w\circ i_{A_1}|_{A_1\cap A_2} = \beta_{\sigma(k)}\circ i_{A_1}|_{A_1\cap A_2} -a_k
	 \end{align} 
	 for any $k=1,\ldots,m$.   This implies the lemma. 
\end{proof}

\subsection{Mutivalued 1-forms and spectral 1-forms}
We prove \cref{main:spectral}, except for \cref{item:unicity}, whose proof is defered to \cref{sec:unicity}.
\begin{thm} \label{thm:spectral}
	Let  $(\overline{X},\Sigma)$ be a smooth log pair. Let $\varrho$, $G$ and $\tilde{u}$  be as in \cref{main:harmonic}. Then 
	\begin{thmlist}
		\item \label{item:extend}the pluriharmonic map $\tilde{u}$ induces a multivalued logarithmic 1-form $\eta$ on the log pair $(\overline{X},\Sigma)$. 
		\item There exists a ramified Galois cover  $\pi:\overline{\xsp}\to \overline{X}$   such that  $\pi^*\eta$ becomes single-valued; i.e., $\pi^*\eta:= \{\omega_1,\ldots,\omega_m\}\subset H^0(\overline{\xsp},\pi^*\Omega_{ \overline{X}}(\log \Sigma))$. 
		\item Denote by $\xsp=\pi^{-1}(X)$, and let $\Sigma_1:=\overline{\xsp}\backslash \xsp$.  Let $\mu:\overline{Y}\to \overline{\xsp}$ be a  log resolution of $(\overline{\xsp},\Sigma_1)$, with $\Sigma_Y:=\mu^{-1}(\Sigma_1)$ a simple normal crossing divisor.   Then  logarithmic forms  $\{\mu^*\omega_1,\ldots,\mu^*\omega_m\}$  are   \emph{pure imaginary}.
	\end{thmlist}
\end{thm}
\begin{proof}
	\noindent {\bf Step 1.} We assume that $G$ is semi-simple.    Let $u$ be the corresponding section of $\widetilde{X}\times_{\varrho }\Delta(G)\to X$ of $\tilde{u}$ defined in \cref{sec:pullback}.   Let $\cR(u)\subset X$ be the regular locus of $u$ defined in \cref{def:sing}. Then $X\backslash\cR(u)$   is an open subset of $X$ of Hausdorff codimension at least two by \cref{gs}.     
	
For any regular point $x\in \cR(u)$ of $u$ (cf.~Definition~\ref{def:sing}),     one can choose a   simply-connected open neighborhood  $U$ of $x$   such that
\begin{enumerate}[label=(\arabic*)]
	\item \label{item:connect} the inverse image $\pi_X^{-1}(U)=\coprod_{\alpha\in I}U_\alpha$ is a union of disjoint open sets in $\widetilde{X}$,  each of which is mapped isomorphically onto 
	$U$ by $\pi_X:\widetilde{X}\to X$.
	\item For some $\alpha\in I$,  there is an apartment $A_\alpha$ of $\Delta(G)$ such that $u({U}_\beta)\subset A_\beta$.
\end{enumerate}
Let $\Phi=\{\beta_1,\ldots,\beta_m\}$ be the subset of $V^*$ defined in \eqref{eq:Phi}. For each apartment $A$ of $\Delta(G) $, $\{\beta_{A,1},\ldots,\beta_{A,m}\}$ are the affine functions on $A$ defined in \eqref{eq:beta}.  For each $j\in \{1,\ldots,m\}$, we define a real function  \begin{align}\label{eq:u}
	 u_{\alpha, j}= \beta_{A_\alpha,j}\circ \tilde{u}\circ  (\pi_X|_{U_\alpha})^{-1} : U\to  \bR.
\end{align}  By the pluriharmonicity of $\tilde{u}$,    we have
 $
\hess u_{\alpha, j} =0
$ 
for each $j$. Hence $\partial u_{\alpha, j}$   is a holomorphic 1-form on $U$.     By \cite[\S 4.2]{BDDM},  the set of holomorphic 1-forms $\{\partial u_{\alpha, 1},\ldots, \partial u_{\alpha, m}\}$  on $U$ will glue together  into  a splitting multivalued 1-forms $\eta$ over $\cR(u)$.  Moreover, for the characteristic polynomial $P_\eta(T):=T^m+\sigma_1 T^{m-1}+\cdots+ \sigma_m$ of $\eta$ defined in \cref{subsec:spectral},  each $\sigma_i$ extends to a logarithmic 1-form in $H^0(\overline{X},\Omega_{ \overline{X}}(\log \Sigma))$. Hence conditions in \cref{prop:spectral} are fulfilled. It implies that,  $\eta$ extends to a multivalued logarithmic 1-form over $(\overline{X},\Sigma)$, and there exists a spectral cover $\pi:\overline{\xsp}\to\overline{X}$ with respect to $\eta$.   The first two assertions of the theorem are proved. 
	
	We denote by $\bar{f}:(\overline{Y},\Sigma_Y)\to (\overline{X},\Sigma)$ be the morphism between log smooth pairs, that is the composition of $\mu$ and $\pi$.  Let $f:Y\to X$ be the restriction of $\bar{f}$ over $Y$.  Then by \cref{main:harmonic}, $\tilde{u}\circ\tilde{f}:\widetilde{Y}\to \Delta(G)$ is an $f^*\varrho$-equivariant pluriharmonic map with logarithmic energy growth. Here we denote by $\tilde{f}:\widetilde{Y}\to \widetilde{X}$ the lift of $f$ between the universal covers. Write $\tilde{v}:=\tilde{u}\circ \tilde{f}$ and let $v$ be the corresponding section.  
	
	We fix an irreducible component $\Sigma_o$ of $\Sigma_Y$.  Since $\cS(u):=X\backslash\cR(u)$ has Hausdorff codimension at least two, we can choose an embedded transverse disk $g:\bD\to \overline{Y}$, such that $g^{-1}(\Sigma_o)=g^{-1}(\Sigma_Y)=\{0\}$, and $g(\bD)$ intersects with  $\Sigma_o$ transversely. Furthermore, $(f\circ g)^{-1}(\cS(u))$ has Hausdorff dimension 0. 
	
	We fix the Euclidean metric $\frac{\sn}{2}dz\wedge d\bar{z}$ over $\bD^*$. By the above construction,  the multivalued 1-forms associated with the equivariant pluriharmonic maps we defined are functorial.  In other words,   $$f^*\eta=\{\mu^*\omega_1,\ldots,\mu^*\omega_m\}\subset H^0(\overline{Y},\Omega_{ \overline{Y}}(\log \Sigma_Y))$$ corresponds to the multivalued 1-form induced by $\tilde{v}$.   Thus, applying \eqref{lem:density} below,  after   rescaling of $\eta$  by multiplying it by $\frac{1}{\sqrt{|W^v|}}$,   we obtain the following over $f^{-1}(\cR(u))$:
	$$
	|\nabla v|^2=\sum_{i=1}^{m}|\mu^*\omega_i+\mu^*\bar{\omega}_i|^2,
	$$
	where  $ |\nabla v|^2$ is  the energy density function of $v$. Since  $ |\nabla v|^2\in L_{\rm loc}^1$, we conclude that the above equality holds over the whole $Y$. 
	
	Let $\nu_g$ be the section of $\bD^*\times_{(f\circ g)^*\varrho}\Delta(G)\to  \bD^*$ defined in \cref{sec:pullback}. On the other hand, since $(f\circ g)^{-1}(\cS(u))$ has Hausdorff dimension 0, by the same argument as above, we can conclude that
	$$
	|\nabla v_g|^2=\sum_{i=1}^{m}|g^*\mu^*\omega_i+g^*\mu^*\bar{\omega}_i|^2. 
	$$
	Write $g^*\mu^*\omega_i=(a_i(z)+\sn b_i(z))d\log z$, where $a_i(z)$ and $b_i(z)$ are real harmonic functions on $\bD$. Then  by the same computation as in \eqref{eq:local} and \eqref{eq:computation},  there exists a constant $C>0$ such that
	\begin{align}\label{eq:computation2}
		8\pi(\sum_{i=1}^{m}|a_i(0)|^2+|b_i(0)|^2)\log\frac{ 1}{r}\leq 	E^{v_g}[\bD_{r,1}] \leq  		8\pi(\sum_{i=1}^{m}|a_i(0)|^2+|b_i(0)|^2)\log\frac{ 1}{r}+C,\quad \forall\ \  r\in (0,1).
	\end{align}
	Let $\gamma\in \pi_1(Y)$ be the element representing the loop $\theta\mapsto  g(\frac{1}{2}e^{\sn\theta})$.  Since $v$ has logarithmic energy growth, by \cref{def:log energy}, we have 
	\begin{align}\label{eq:translation com}
		L_{f^*\varrho(\gamma)}^2=16\pi^2(\sum_{i=1}^{m}|a_i(0)|^2+|b_i(0)|^2). 
	\end{align}
	 	Since $(f\circ g)^{-1}(\cS(u))$ has Hausdorff dimension 0,  by \cite[Corollary 1]{Shi68} there exists a   subset $I\subset (0,1)$ of Lebesgue measure 1, such that  for each $r\in I$, the loop $\ell_r$ in $\bD^*$  defined by $\theta\mapsto re^{\sn\theta}$, does not intersect with $(f\circ g)^{-1}(\cS(u))$. Let $\gamma\in \pi_1(Y)$ be the element representing the loop $\theta\mapsto  g(re^{\sn\theta})$.  In this case, the translation length $L_{f^*\varrho(\gamma)}$ satisfies that, for any $r\in I$, we have
	\begin{align*}
		L_{f^*\varrho(\gamma)}&\leq \oint_{\ell_r} \sqrt{\sum_{i=1}^{m}|(g^*\mu^*\omega_i+g^*\mu^*\bar{\omega}_i)(\frac{\partial}{\partial\theta})|^2} d\theta \\
		&=\int_{0}^{2\pi} \sqrt{\sum_{i=1}^{m}|2b_i(re^{\sn\theta})|^2}d\theta.
	\end{align*}
	If letting $r\in I$ tends to $0$, we have
	$$
	L_{f^*\varrho(\gamma)}^2\leq 16\pi^2 \sum_{i=1}^{m}b^2_i(0).
	$$
	It follows from \eqref{eq:translation com} that $a_i(0)=0$ for each $i$. Thus, for each $i\in \{1,\ldots,m\}$, we have  ${\rm Res}_{\Sigma_o}\mu^*\omega_i=\sn b_i(0)$, which is pure imaginary.   Since $\Sigma_o$ is an arbitrary irreducible component of $\Sigma_Y$, it follows that $\mu^*\omega_1,\ldots,\mu^*\omega_m$ are pure imaginary logarithmic forms. The theorem is thus proved when $G$ is semisimple. 
	
	\medspace
	
	\noindent {\bf Step 2.} 
	We assume that $G$ is reductive. We shall use the notation introduced in the proof of \cref{main:harmonic} without recalling them explicitly.  Recall that $\Delta(G) = \Delta(\sD G) \times V$, where $V$ is isometric to $\bR^N$, and $G(K)$ acts on $V$ by translation. Note that $\tilde{u}$ is the product of a $\sigma$-equivariant pluriharmonic map $\tilde{u}_0: \widetilde{X} \to \Delta(\sD G)$ with logarithmic energy growth, and a $\tau$-equivariant pluriharmonic map $\tilde{v} \circ \tilde{a}: \widetilde{X} \to V$, also with logarithmic energy growth. Thus, the multivalued 1-form $\eta$ induced by $\tilde{u}$ is merely the union of the multivalued 1-form $\eta_0$ induced by $\tilde{u}_0$,   and the logarithmic 1-forms $\{\zeta_1, \ldots, \zeta_k\} \subset H^0(\overline{X}, \Omega_{\overline{X}}(\log \Sigma))$ induced by $\partial (\tilde{v} \circ \tilde{a})$. Hence,   the spectral cover $\pi: \overline{\xsp} \to \overline{X}$ with respect to $\eta$ coincides with the spectral cover with respect to $\eta_0$, whose existence is ensured by Step 1. The first two items are proved. 
	
	By Step 1, $f^*\eta_0$ is a finite set of pure imaginary logarithmic 1-forms.  Recall that in \cref{main:harmonic}, we prove that  $\{\zeta_1, \ldots, \zeta_k\} $  are pure imaginary. Thus, $f^*\zeta_i$ is also pure imaginary for each $i$. We conclude that $f^*\eta=f^*\eta_0\cup \{f^*\zeta_1,\ldots,f^*\zeta_k\}$ is a set of pure imaginary logarithmic 1-forms. This completes the proof of the theorem. 
\end{proof}
By the proof of \cref{item:extend},  if $G$ is semi-simple,  at each point $x_0$ of $\cR(u)$, there exists an open neighborhood $U$ of $x_0$ such that,  $\eta$ is given by holomorphic 1-forms  $\{\partial u_{\alpha,1},\ldots,\partial u_{\alpha,m}\}$, where $u_{\alpha,j}:U\to \bR$ is defined in \eqref{eq:u}.  Let $U_\alpha$ be a connected component of $\pi_X^{-1}(U)$ introduced in \cref{item:connect}.  Note that
$$
|\nabla u|^2=\sum_{i=1}^{N}|dx_i\circ i_{A_\alpha}\circ \tilde{u}\circ (\pi_X|_{U_\alpha})^{-1} |^2,
$$
where $\{x_1,\ldots,x_N\}$ is some orthogonal coordinates for $V$ defined in \cref{subsec:building}. For any $w\in W$, note that $\{w^*dx_1,\ldots,w^*dx_N\}$  is a   orthogonal basis for $TV^*$. Hence, by the definition of $\Phi$ defined in \eqref{eq:Phi}, we have
\begin{align} \nonumber 
	 \sum_{j=1}^{m}|\partial u_{\alpha,j}|&=\sum_{w\in W^v}\sum_{i=1}^{N}|w^*\partial x_i\circ i_{A_\alpha}\circ  \tilde{u}\circ (\pi_X|_{U_\alpha})^{-1} |^2\\\nonumber
	& =|W^v|\cdot \sum_{i=1}^{N}|\partial x_i\circ i_{A_\alpha}\circ \tilde{u}\circ (\pi_X|_{U_\alpha})^{-1} |^2\\ \label{eq:density}
	& =\frac{|W^v|}{2}|\nabla u|^2.
\end{align}  
The following result will be used in \cref{sec:singular}.
\begin{lem}\label{lem:5}
Let $X$, $G$, $\varrho$ and $\tilde{u}$ be as in \cref{main:harmonic}. Then there exists a   multivalued logarithmic 1-form $\eta$ on $(\overline{X},\Sigma)$, that is splitting over $\cR(u)$, such that for any point $x\in \cR(u)$, it has a simply connected open neighborhood $U$ satisfying:
\begin{thmlist}
	 \item \label{lem:density} over $U$,   $\eta$ is represented by  some  holomorphic 1-forms $\{\omega_1,\ldots,\omega_{N\ell}\}$ on $\Omega$, and  
\begin{align}\label{eq:energy3}
	  |\nabla u|^2=2\sum_{j=1}^{N\ell}|\omega_{j}|^2,
\end{align} 
	 where $N$ is the $K$-rank of $ G$, and $\ell$ is the cardinality of the vectorial Weyl group $W^v$ of $\sD G$. 
	 \item \label{item:ortho} There exists a  partition  of  $\sqcup_{i=1}^{\ell}\{\omega_{i,1},\ldots,\omega_{i,N}\}= \{\omega_1,\ldots,\omega_{\ell N}\}$  satisfying
	 \begin{itemize}
	 	\item  for each $i=2,\ldots,\ell$, there exists a constant matrix $M_i\in {\rm O}(N,\bR)$ such that
\begin{align}\label{eq:orth}
	 	 \begin{bmatrix}
	 	\omega_{i,1}, & \cdots, &\omega_{i,N}
	 \end{bmatrix} =	 	 \begin{bmatrix}
	 	\omega_{1,1}, & \cdots, &\omega_{1,N}
	 \end{bmatrix}  \cdot M_i.
\end{align} 
	 	\item If  there exists some apartment $A$ of $\Delta(G)$,  such that  $\tilde{u}(U_\alpha)\subset A$, where $U_\alpha$ is some connected component of $\pi_X^{-1}(U)$, then  for   any isometry  $i:A\to \bR^N$, denoting $(u_1,\ldots,u_N)=i\circ \tilde{u}\circ (\pi_X|_{U_\alpha})^{-1}:U\to \bR^N$,  we have 
\begin{align}\label{eq:orth2}
	 	\begin{bmatrix}
	 	\partial u_1, & \cdots,  & \partial u_N
	 \end{bmatrix} =	 \begin{bmatrix}
	 	\omega_{1,1}, & \cdots, &\omega_{1,N}
	 \end{bmatrix} \cdot M \cdot  \frac{1}{\sqrt{\ell}}
\end{align} 
	 	for some constant matrix $M\in {\rm O}(N,\bR)$.
	 \end{itemize} 
	 \item \label{item:p}For each $p\in \{1,\ldots,n\}$, $\eta$ induces a multivalued section $\eta^p$ on $\Omega^p_{ \overline{X}}(\log \Sigma)$.
\end{thmlist} 
\end{lem}
\begin{proof}
We shall use the notations in Step 2 of the proof of \cref{thm:spectral}.  For   each point $x_0$ of $\cR(u_0)$, there exists a simply connected neighborhood $U$ of $x_0$ such that, for some connected component $U_\alpha$ of $\pi_X^{-1}(U)$,    $\tilde{u}_0(U_\alpha)$ is contained in some apartment $A$ of $\Delta(\sD G)$.  Let $(W^v,V)$ be data of $\sD G$  defined in \cref{subsec:building}.  
Let $N'$ be the dimension of $\Delta(\sD G)$ and $\ell$ be the cardinality of $ W^v$.  Fix  orthonormal coordinates $(x_1,\ldots,x_{N'})$ for $V$.

We use the notations in \cref{subsec:building}.  Define a set of holomorphic 1-forms on $U$  with a partition as follows:
\begin{align} \label{eq:partition}
\sqcup_{w\in W^v}\{ \frac{1}{\sqrt{\ell}}w^*\partial x_1\circ i_{A}\circ\tilde{u}_0\circ  (\pi_X|_{U_\alpha})^{-1},\ldots, \frac{1}{\sqrt{\ell}}w^*\partial x_{N'}\circ i_{A}\circ\tilde{u}_0\circ  (\pi_X|_{U_\alpha})^{-1}, \frac{1}{\sqrt{\ell}}\xi_{1},\ldots, \frac{1}{\sqrt{\ell}}\xi_{k}\},
\end{align}
 where $\xi_1,\ldots,\xi_k$ are logarithmic 1-forms on $(\overline{X},\Sigma)$ induced by the pluriharmonic map $\tilde{v}\circ \tilde{a}$. Thus we have $N=N'+k$, as $N$ is also the dimension of $\Delta(G)$.   By Step 1 of the proof of \cref{item:extend}, \eqref{eq:partition} gives rise to a   multivalued logarithmic 1-form on $(\overline{X},\Sigma)$, denoted by $\eta$.
By \eqref{eq:density}, we have
$$
|\nabla u_0|^2=2\sum_{i=1}^{N'}\frac{1}{ {|W^v|}} \left|  \partial x_{1}\circ i_{A}\circ\tilde{u}_0\circ  (\pi_X|_{U_\alpha})\right|^2+\cdots+\left|  \partial x_{N'}\circ i_{A}\circ\tilde{u}_0\circ  (\pi_X|_{U_\alpha})\right|^2
$$  
 Note that $\partial \tilde{v}\circ \tilde{a}=(\pi_X^*\xi_1,\ldots,\pi_X^*\xi_k)$. Since $|\nabla u|^2=|\nabla u_0|^2+|\nabla \tilde{v}\circ \tilde{a}|^2$,  it yields \eqref{eq:energy3}. 
 
   Note that $w$ is an isometry of $V$. 
\cref{item:ortho} follows directly from the construction of $\eta$  in \eqref{eq:partition}. 

Let us prove \cref{item:p}. For each $I=\{i_1,\ldots,i_p\}$ with $1\leq i_1<\cdots<i_p\leq n$, we define  a set of holomorphic $p$-forms with a partition given by 
$$
\sqcup_{j=1}^{\ell}\{\pm \omega_{j,i_1}\wedge\cdots\wedge \omega_{j,i_p} \}_{1\leq i_1<\cdots<i_p\leq N}.
$$
By \eqref{eq:partition}, this is a well-defined  splitting multivalued $p$-form on $\cR(u)$, denoted by $\eta^p$.

	We fix a smooth hermitian metric $h$ for  the vector bundle $\Omega_{ \overline{X}}(\log \Sigma)$. It induces a hermitian metric $h_p$ on $\Omega^p_{ \overline{X}}(\log \Sigma)$.  Since the support $|Z_\eta|$ is compact,  there exists a uniform  constant $C>0$ such that   
\begin{align*}
	\left| \omega_{j,i_1}\wedge\cdots\wedge \omega_{j,i_p}(x)\right|_{h^p}\leq C,\quad \forall \  x\in U\cap \cR(u) 
\end{align*} 
for each $I$. 
Let $P_{\eta^p}(T)=T^M+\sigma_1 T^{M-1}+\cdots+\sigma_M$ be the characteristic polynomial of $\eta^p$ defined in \cref{subsec:spectral}, with $\sigma_i\in H^0(\cR(u),{\rm Sym}^i\Omega^p_{ \overline{X}}(\log \Sigma)|_{\cR(u)})$.   Then the norm of $\sigma_i$ with respect to the metric   $h^p$ is uniformly bounded. By the Hartogs theorem in \cite{Shi68}, each $\sigma_i$ extends to a section of ${\rm Sym}^i\Omega^p_{ \overline{X}}(\log \Sigma)$ on $\overline{X}$. The conditions in \cref{prop:spectral} are fulfilled. We conclude that $\phi^p$ extends to a multivalued section of $\Omega^p_{ \overline{X}}(\log \Sigma)$ on $\overline{X}$.The last assertion is proved. 
\end{proof}

\begin{rem} 
When $X$ is a compact K\"ahler manifold, spectral covers associated with equivariant harmonic maps to Euclidean buildings were   systematically studied  by Eyssidieux in \cite{Eys04}. The construction of spectral covers presented here follows the approach of Klingler \cite{Kli13}, while the  definition of multivalued 1-forms builds on the ideas of \cite{Eys04}, which differ slightly from those in \cite[\S 4]{BDDM}.
\end{rem}

\section{Unicity of pluriharmonic maps} \label{sec:unicity}

\subsection{Uniqueness of  energy density function}\label{sec:unique}
Throughout this subsection,  $G$ is a reductive algebraic group defined over a non-archimedean local field $K$. We begin with the following definition. 
\begin{dfn}[Directional energy]
	Let $u:\Omega \rightarrow \Delta(G)$ be a locally finite energy map from a Riemannian domain $\Omega$.
	For $V \in \Gamma(\Omega, T_\Omega)$, the \emph{directional energy} defined in \cite[Theorem 1.9.6]{KS} is denoted by  $|u_*(V) |^2$.  By \cite[Lemma 1.9.3 and Theorem 2.3.2]{KS}, 
	\[
	|u_*(V)|^2(p)=\lim_{t \rightarrow 0} \frac{d^2(u(p),u(\exp_p(tV)))}{t^2} \ \ \mbox{for a.e.~$p \in \Omega$}.
	\]
\end{dfn}

\begin{rem}
	Let $M$ be a Riemannian manifold,  $\varrho:\pi_1(M) \rightarrow G(K)$ be a representation and $u:\widetilde M \rightarrow \Delta(G)$ be a $\varrho$-equivariant  map.  Given a vector field $V$ defined on $M$,  lift it to $\tilde M$ and denote it again by $V$.  Then the energy density function $|u_*(V)|^2$ is a $\pi_1(M)$-invariant function on $\tilde M$ and thus descends to a well-defined $L^1_{\rm loc}$-function on $M$.
\end{rem}
\begin{proposition} \label{prop:uv}
	Let $X$ be a smooth quasi-projective variety  of dimension $n$ and  $\varrho:\pi_1(X) \rightarrow G(K)$ be a  representation.  If $\tilde{u},\tilde{v}: \widetilde{X} \rightarrow \Delta(G)$ are two $\varrho$-equivariant pluriharmonic maps of logarithmic energy growth, then we have
	\begin{thmlist}
		\item  $d(\tilde{u},\tilde{v})=c$ for some constant $c\geq 0$;
		\item  \label{item:density}$|\tilde{u}_*(V)|^2 = |\tilde{v}_*(V)|^2$ for any holomorphic vector field $V \in \Gamma(\Omega, T_\Omega)$, where $\Omega\subset \widetilde{X}$ is an open set.  
	\end{thmlist} 
\end{proposition}

\begin{proof}
	If $\dim_\bC X=1$, then the proposition  follows from \cite[Lemma 5.8]{DMrs}.
	Assume by induction  that the assertions are both true  if  $\dim X=n-1$.   We take a smooth projective compactification $\overline{X}$ for $X$ such that $\Sigma:=\overline{X}\backslash X$ is a simple normal crossing divisor.  
	
	We fix an projective embedding $\iota:\overline{X}\hookrightarrow \bP^N$ and denote by $L:=\iota^*\sO_{\bP^N}(3)$.  Let $\bU(q)\subset H^0(\overline{X},L)$ be defined in \cref{prop:Bertini}.   
	For any element $s \in H^0(\overline{X},L)$, let  $\iota_{Y_s}:Y_s \rightarrow X$ be  the inclusion map defined in \cref{prop:Bertini}. 
	
	Choose any $q\in X$, and any $\tilde{q}\in \widetilde{X}$ such that $\pi_X(\tilde{q})=q$.   By \cref{item:Lefschetz}, for  any section $s\in \bU(q)$,  letting   $\widetilde{\iota_{Y_{\! s}}}:\widetilde{Y_{\! s}}\to \widetilde{X}$ be the lift of $\iota_{Y_{\! s}}$ between universal covers, we have 
	$
	\pi_X^{-1}(q)\subset  \widetilde{\iota_{Y_{\! s}}}(\widetilde{Y}_s).$ Hence there exists $\widetilde{q_s}\in \widetilde{Y_s}$ such that $\widetilde{\iota_{Y_{\! s}}}(\widetilde{q_s})=\tilde{q}$.    By \cref{main:harmonic},  the  $\varrho_{Y_{\! s}}$-equivariant maps $\widetilde{u_{Y_{\! s}}}$ and $\widetilde{v_{Y_{\! s}}}$  defined in \cref{sec:pullback} are pluriharmonic maps  of logarithmic energy. The inductive hypothesis implies that there exists a constant $c_{Y_{\! s}} \geq 0$ such that $d\left(\widetilde{u_{Y_{\! s}}}(y), \widetilde{v_{Y_{\! s}}}(y)\right) = c_{Y_{\! s}}$ for each $y\in \widetilde{Y_{\! s}}$.   Since $\widetilde{u_{Y_{\! s}}} =\tilde{u}\circ \widetilde{\iota_{Y_{\! s}}}$  and  $\widetilde{v_{Y_{\! s}}} =\tilde{v}\circ \widetilde{\iota_{Y_{\! s}}}$, it follows that for any other $s^{\prime} \in \mathbb{U}(q)$, we have
	$$
	c_{Y_{\! s}}=d\left(\widetilde{u_{Y_{\! s}}}(\tilde{q}_s), \widetilde{v_{Y_{\! s}}}(\tilde{q}_s)\right)= d\left(\tilde{u}(\tilde{q}),\tilde{v}(\tilde{q})\right)=d\left(\widetilde{u_{Y_{s'}}}(\tilde{q}_{s'}), \widetilde{v_{Y_{s'}}}(\tilde{q}_{s'})\right) = c_{Y_{s^{\prime}}}.
	$$ 
	Hence $c_{Y_{\! s}}$ does not depend on the choice of $s\in \bU(q)$, which we shall denote by $c$.
	
	Let $p$ be any other point in $X$. Then by \cref{item:tangent}, there exists $s\in \bU(q)$ such that $p\in Y_s$.  By \cref{item:Lefschetz}, for any   $\tilde{p}\in \pi_X^{-1}(p)$,     there exists $\widetilde{p_s}\in \widetilde{Y_s}$ such that $\widetilde{\iota_{Y_{\! s}}}(\widetilde{p_s})=\tilde{p}$. It follows that
	$$
	d\left(\tilde{u}(\tilde{p}),\tilde{v}(\tilde{p})\right)=d\left(\widetilde{u_{Y_{\! s}}}(\tilde{p}_s), \widetilde{v_{Y_{\! s}}}(\tilde{p}_s)\right)= c.
	$$
	Thus, we conclude that $d(\tilde{u}(x), \tilde{v}(x)) \equiv c$  for  each $x\in \widetilde{X}$. 
	
	\medspace

	Let us prove the second assertion. For any local smooth vector field $V$ on $\widetilde{X}$, we know that $|\tilde{u}_*(V)|^2,|\tilde{v}_*(V)|^2\in L^1_{\rm loc}$, and thus it suffices to prove \cref{item:density} over the dense open subset $\cR(\tilde{u})\cap \cR(\tilde{v})$.  Since $\tilde{u}$ and $\tilde{v}$ are both smooth over  $\cR(\tilde{u})\cap \cR(\tilde{v})$, it suffices to prove that for any  point $\tilde{q}\in \cR(\tilde{u})\cap \cR(\tilde{v})$, and any $V\in T_{\tilde{q}}\widetilde{X}$, we have
	$$
	|\tilde{u}_*(V)|^2=|\tilde{v}_*(V)|^2.
	$$ Set $q=\pi_X(\tilde{q})$.  By \cref{item:tangent}, there exists $s\in \bU(q)$ such that $(\pi_X)_*V\in T_{q}Y_s$.   Hence $V\in T_{\tilde{q}}(\widetilde{Y_s})$. 
	
	By the inductive hypothesis, we have
	$$|\tilde{u}_*(V)|^2=|(\widetilde{u_{Y_s}})_*(V)|^2=|(\widetilde{v_{Y_s}})_*(V)|^2=|\tilde{v}_*(V)|^2.$$  This yields the second assertion. The proposition is proved. 
\end{proof}

\subsection{Proof of  unicity theorem }\label{subsec:unicity}
Recall the following definition from \cite{GS92}.
\begin{dfn}[\cite{GS92}, Section 6]
	We say that a nonpositively curved $N$-dimensional complex $\cF$ is {\it $F$-connected} if any two adjacent simplices are contained in an apartment, i.e.~a totally geodesic subcomplex $A$ which is isometric to a subset of the Euclidean space $\mathbb R^N$.
\end{dfn}

The regular set and the singular set of a harmonic map into a $F$-connected complex is defined analogously as in \cref{def:sing}.

A neighborhood of a point $P_0 \in \Delta(G)$ is isometric to a neighborhood of the origin in the tangent cone $T_{P_0}\Delta(G)$.  Two simplices (which are actually simplicial cones) in $T_{P_0}\Delta(G)$ are contained in a totally geodesic subcomplex $T_{P_0}A$ where $A$ is an apartment of $\Delta(G)$. In other words, $T_{P_0}\Delta(G)$ is  an  $N$-dimensional $F$-connected  complex.
Thus, when we study the local behavior of harmonic maps $u:\Omega \to \Delta(G)$ at a point $x_0 \in \Omega$, we can assume that $u$ maps into the $N$-dimensional, $F$-connected complex   $T_{P_0}\Delta(G)$ where $P_0=u(x_0)$.


\begin{lem}[\cite{GS92}, proof of Proposition 2.2]  \label{order}
	Let $u:\Omega \rightarrow \cF$ be a harmonic map from an $n$-dimensional Riemannian domain to a $F$-connected complex and $x_0 \in \Omega$.  Then  there exists a constant $c>0$ and $\sigma_0>0$ such that 
	\[
	\sigma \rightarrow \frac{e^{c\sigma^2} \sigma \displaystyle{ \int_{B_\sigma(x_0)} |\nabla u|^2 d\mu} }
	{\displaystyle{ \min_{Q \in \Delta(G)} \int_{\partial B_\sigma(x_0)} d^2(u,Q) d\Sigma}}
	\]
	is a  non-decreasing functions in the interval $(0,\sigma_0)$.\qed
\end{lem}

\begin{dfn}\label{def:ord}
	For $u$ and $x_0$ as in Lemma~\ref{order}, we set
	\[
	{\rm Ord}^u(x_0) = \lim_{\sigma \rightarrow 0} \frac{e^{c\sigma^2} \sigma \displaystyle{ \int_{B_\sigma(x_0)} |\nabla u|^2 d\mu} }
	{\displaystyle{\min_{Q \in \Delta(G)} \int_{\partial B_\sigma(x_0)} d^2(u,Q) d\Sigma}}.
	\]
\end{dfn} 
As a limit of non-decreasing sequence of functions, $x \mapsto {\rm Ord}^u(x)$ is a upper semicontinuous function.  Thus, we have the following:
\begin{enumerate}[label=(\alph*)]
	\item By \cite[Lemma 1.3]{GS92}, ${\rm Ord}^u(x) \geq 1$ for all $x \in \Omega$.
	\item By \cite[Theorem 6.3.(i)]{GS92},  if $x_i \rightarrow x$ and ${\rm Ord}^u(x_i)>1$, then ${\rm Ord}^u(x) >1$. 
\end{enumerate} 
\begin{lem}[\cite{GS92},  proof of Theorem 6.4]  \label{snot}
	Let $u$ be as in \cref{order} and $\tilde{\mathcal S}_0(u)$ to be the set of points $x \in \Omega$ such that ${\rm Ord}^u(x)>1$.  Then $\tilde{\mathcal S}_0(u)$ is a closed set such that $\dim_{\mathcal H}(\tilde{\cS}_0(u)) \leq n-2$. \qed
\end{lem}

\begin{lem}[\cite{GS92}, proof of Proposition 2.2, Theorem 2.3] \label{lem23}
	Let $u$ and $x_0$ be as in Lemma~\ref{order} and let $\alpha:={\rm Ord}^u(x_0)$. There exists a constant $c>0$ and $\sigma_0>0$ such that 
	\[
	\sigma \rightarrow \frac{e^{c\sigma^2}}{\sigma^{n-1+2\alpha}} \int_{\partial B_\sigma(x_0)} d^2(u,u(x_0)) d\Sigma
	\]
	and
	\[
	\sigma \rightarrow \frac{e^{c\sigma^2}}{\sigma^{n-2+2\alpha}} \int_{B_\sigma(x_0)} |\nabla u|^2 d\mu
	\]
	are non-decreasing functions in the interval $(0,\sigma_0)$. \qed
\end{lem}

\begin{rem} \label{rem:energydensity}
	For a finite energy map $u:\Omega \rightarrow \cF$ into a $F$-connected complex,   $|\nabla u|^2\in L^1_{\rm loc}$ is not necessarily defined at all points of $\Omega$.  On the other hand,  it follows from Lemma~\ref{lem23} that 
	for a harmonic map $u$, 
	we can define $|\nabla u|^2$ at every point of $x_0 \in \Omega$ by setting
	\[
	|\nabla u|^2(x_0) = \lim_{\sigma \rightarrow 0} \frac{1}{c_n \sigma^n}
	\int_{B_\sigma(x_0)} |\nabla u|^2 \, d\mu
	\] 
	where $c_n \sigma^n$ is the volume of a ball or radius $\sigma$ in Euclidean space.  
\end{rem}

Let $u:\Omega \rightarrow \Delta(G)$ be a harmonic map and  $x_0 \in \Omega$.  Use normal coordinates centered at $x_0$ to identify $x_0=0$ and let  $\bB_r(0)=\{x=(x^1, \dots, x^n) \in \bR^n:  |x|<r\}$.  
As mentioned above, we can identify a neighborhood of $u(0)$ with a neighborhood of the origin $\cO$ of the tangent cone $T_{u(0)}\Delta(G)$.  For $\mu>0$ and $P \in T_{u(0)}\Delta(G)$, denote by $\mu P$ to be the point in $T_{u(0)}\Delta(G)$ on the geodesic ray emanating from $\cO$ and going through $P$ at a distance $\mu d(\cO, P)$ from $\cO$.  
Let 
\[
\mu(\sigma)= \left(\sigma^{1-n} \int_{\partial B_\sigma(0)} d^2(u,u(0)) d\Sigma \right)^{-1}.
\]

\begin{dfn} \label{blowups} The {\it blow up map} is defined by
	\[
	u_\sigma:\bB_1(0) \rightarrow T_{u(0)}\Delta(G), \ \ \ u_\sigma(x)=\mu(\sigma) u(\sigma x).
	\]
\end{dfn}

By \cite[Proposition 3.3]{GS92} and the paragraph proceeding it, there exists a sequence $\sigma_i \rightarrow 0$ such that  $u_{\sigma_i}$ converges locally uniformly to a non-constant homogeneous harmonic map $u_*$ of degree $\alpha:={\rm Ord}^u(x_0)$.

If ${\rm Ord}^u(x_0)=1$, then have the following: 
\begin{enumerate}[label=(\alph*)]
	\item  By~\cite[Proposition 3.1]{GS92},  there exists $m \in \{1, \dots, \min\{ n,N\} \}$ such that 
	\[
	u_*=J \circ v\big|_{B_1(0)}
	\]
	for  an isometric and totally geodesic  embedding $J:\bR^m \rightarrow T_{u(0)} \Delta(G)$ and   a  linear map $v:\bR^n \rightarrow \bR^m$ of full rank.  
	\item By \cite[Lemma 6.2]{GS92}, the union of all $N$-flats of $T_{u(0)} \Delta(G)$ containing $J(\bR^m)$ is isometric to   $\bR^m \times \cF$ where $\cF$ is a $(N-m)$-dimensional, $F$-connected complex.
	\item \label{item:splitting}By  \cite[Theorem 6.3]{GS92}, there exists $\sigma_0>0$ such that $u(B_{\sigma_0}(x_0)) \subset  \bR^m \times \cF$.  If we write 
	\begin{equation} \label{product}
		u=(u^1, u^2): B_{\sigma_0}(x_0) \rightarrow \bR^m \times \cF,
	\end{equation}
	then $u^1: B_{\sigma_0}(x_0) \rightarrow \bR^m$ is a smooth harmonic map of rank $m$ and $u^2:B_{\sigma_0}(x_0) \rightarrow \cF$ is a harmonic map  with $\alpha_2:={\rm Ord}^{u^2}(x_0) \geq 1+\epsilon$  for $\epsilon>0$. 
\end{enumerate}

%

	\begin{lem} \label{ord1}
		Let $u$ and $x_0$ be as in Lemma~\ref{order}.  Then
		\[
		{\rm Ord}^u(x_0) >1 \ \ \Leftrightarrow \ \ |\nabla u|^2(x_0) =0.
		\]
	\end{lem}
	
	\begin{proof}
		First, assume $\alpha:= {\rm Ord}^u(x_0) >1$.  
		Lemma~\ref{lem23} implies that there exists a constant $C>0$ and $\sigma_0>0$ such that for $\sigma \in (0,\sigma_0)$ \[
		\int_{\partial B_\sigma(x_0)} d^2(u,u(x_0)) d\Sigma \leq C \sigma^{n-1+2\alpha}.
		\]
		By 
		\cref{rem:energydensity}, the above inequality and $\alpha>1$ imply (with $c_n$ equal to the volume of the unit ball in $\bR^n$)
		\begin{equation} \label{secgrad}
			|\nabla u|^2(x_0) = \lim_{\sigma \rightarrow 0}  \frac{1}{c_n\sigma^n} \int_{B_\sigma(x_0)} |\nabla u|^2 \, d\mu =    \lim_{\sigma \rightarrow 0}  \frac{1}{c_n \sigma^{n+1}} \int_{\partial B_\sigma(x_0)} d^2(u,u(x_0)) d\Sigma=0
		\end{equation}
		Next, assume ${\rm Ord}^u(x_0)=1$.  Use normal coordinates centered at $x_0$  and write $u=(u^1, u^2)$ as in \cref{product}.  Define ${_\theta u}=({_\theta u}^1, {_\theta u}^2)$ by setting ${_\theta u}(x)= \theta^{-1} u(\theta x)$.  From \cite[(5.14)]{GS92}, ${_\theta u} \rightarrow L$ uniformly on compact subsets  to a non-constant homogeneous degree 1 map  $L$.  Furthermore, since $\alpha_2:={\rm Ord}^{u^2}(x_0) >1$ (cf.~(c)),  arguing analogously as (\ref{secgrad}), we get
		\[
		\lim_{\theta \to 0}\int_{\partial B_1(0)} d^2({_\theta u}^2, {_\theta u}^2(0)) d\Sigma = \lim_{\theta \to 0} \frac{1}{\theta^{n+1}} \int_{\partial B_\theta (0)} d^2(u^2, u^2(0)) d\Sigma = \lim_{\theta \to 0} C\theta^{2\alpha_2-2}=0.
		\]
		By the maximum principle, this implies that ${_\theta u}^2 \to {_\theta u}^2(0) =u^2(0)$ uniformly on compact subsets of $B_1(0)$.  This in turn implies that ${_\theta u}^1 \to L$ uniformly on compact subsets of $B_1(0)$.
		Since   ${_\theta u}^1$ is a smooth harmonic map, ${_\theta u}^1 \to L$  in $C^k$ for any $k$ in any compact subset of $B_1(0)$.  
		Since 
		$
		|\nabla L|^2(0) > 0$,  we also have $  |\nabla u^1|^2(0) = |\nabla {_\theta u}^1|^2(0)> 0$.  Therefore,  $|\nabla u|^2 > 0$.   
	\end{proof} 

\begin{dfn}  \label{dfn:M}
	Let $u:\Omega \rightarrow \Delta(G)$ be a harmonic map.  For $x_0 \in \tilde \cS_0(u)$, define $m_{x_0}=0$.  For $x_0 \in \Omega \backslash \tilde \cS(u)$, let $m_{x_0}$ be  the integer $m$ in \cref{product}.   Let
	$M:=\sup_{x_0 \in \Omega} m_{x_0}$.  We say the  point $x_0 \in \Omega$ is a  \emph{critical point} if $m_{x_0} < M$.  We denote the set of critical points by $\tilde \cS(u)$.  Define $\tilde \cR(u) = \Omega \backslash \tilde \cS(u)$.
\end{dfn}

\begin{lem}\label{lem:compl}
	If $u:\Omega \rightarrow \Delta(G)$ is a non-constant harmonic map, then $\tilde \cR(u) \subset \cR(u)$.
\end{lem}

\begin{proof}
	Let $x_0 \in \tilde \cR(u)$; i.e.~$m_{x_0}=M$ where $M$ is as in Definition~\ref{dfn:M} and there exists $\sigma>0$ such that we can write 
	\begin{equation} \label{prod}
		u=(u^1,u^2): B_\sigma(x_0) \rightarrow \bR^M \times \cF.
	\end{equation}  
	By choosing $\sigma>0$ smaller if necessary,  we can assume that $u^1$ is of rank $M$ at all points $x \in B_{\sigma_0}(x_0)$.  Therefore, the restriction of \eqref{prod} to $B_r(x)$ is an expression of $u$ as $u=(u_1,u_2)$ as in \cref{product} in $B_r(x) \subset B_{\sigma_0}(x_0)$.  By \cref{secgrad}, $|\nabla u^2|^2(x) =0$ for all $x \in B_{\sigma_0}(x_0)$.  Thus, we conclude that $u^2 \equiv P_0$ for some $P_0 \in \cF$. 
	Hence $u(B_\sigma(x_0)) \subset \bR^M \times \{P\}$ which implies  $B_\sigma(x_0) \subset  \cR(u)$. 
\end{proof}

\begin{lem}
	The set of critical points $\tilde \cS(u)$ is a closed set of Hausdorff dimension at most $n-2$. 
\end{lem}

\begin{proof}
	By \cref{lem:compl},  $\tilde \cS(u) = \cS(u) \cup (\tilde \cS(u) \cap \cR(u))$.   By \cite[Theorem 6.4]{GS92}, $\cS(u)$ is a closed set of Hausdorff dimension at most $n-2$.   Thus, the assertion follows from the fact that the Hausdorff dimension of the set of critical points of a harmonic map into Euclidean space is at most $n-2$. 
\end{proof}

Note that $\Delta(G)$ is a simplicial complex.  For $k=0,\dots,N$, denote the $k$-skeleton of $\Delta(G)$ by $\Delta(G)^{(k)}$.  
\begin{proposition} \label{omegastar}
Let  $u: \Omega \rightarrow \Delta(G)$ be a non-constant harmonic map from  some connected open domain $\Omega$ in $\bC^n$, that is also pluriharmonic.  Then there exists a positive integer $k$ such that,
\begin{thmlist}
	\item\label{item:subspace}  for any $x\in \cR(u)$, there exists an apartment $A$ and  an open neighborhood $\Omega_x$ of $x$ such that    $u(\Omega_x)\subset L\subset A\cap \Delta(G)^k$, where $L$ is some  $k$-dimensional affine  subspace in $A\cap \Delta(G)^{k}$.
	\item \label{item:notcontain}    $u(\Omega_x)\not\subset \Delta(G)^{k-1}$.   
\end{thmlist} 
Moreover, 	let 
	$
	\Omega^*(u)$ be the set of points $x \in \Omega$ such that there exists $r>0$ and a chamber $C$ of $\Delta(G)$ such that $u(B_r(x)) \subset \overline{C}$.  Then $\Omega^*(u)$ is an open set of full measure in $\Omega$.
\end{proposition}

\begin{proof}   We first prove the following claim.
\begin{claim}\label{claim:contain}
	Let $u: \Omega_0 \to \bR^N$ be a pluriharmonic map from a connected open domain $\Omega_0 \subset \bC^n$, and suppose that for some open subset $U \subset \Omega_0$, we have $u(U) \subset \bR^\ell \times \{(0,\ldots,0)\}$. Then $u(\Omega_0) \subset \bR^\ell \times \{(0,\ldots,0)\}$.
\end{claim}
\begin{proof}
	Write $u = (u_1, \ldots, u_N)$. By the assumption, for each $i \in \{\ell+1, \ldots, N\}$, we have $u_i(x) = 0$ for all $x \in U$. Since $u_i$ is a harmonic function, it follows that $u_i \equiv 0$ for each $i \in \{\ell+1, \ldots, N\}$. This proves the claim.
\end{proof} 
\begin{claim}\label{claim:smaller}
Assume that  $\cN$ is a  connected open subset of $\cR(u)$ such that there exists an apartment $A$ of $\Delta(G)$ such that $u(\cN) \subset A \cap \Delta(G)^k$ and $u(\cN) \not\subset \Delta(G)^{k-1}$. Then there exists a $k$-dimensional affine subspace in $A \cap \Delta(G)^k$ that contains the image $u(\cN)$. 
\end{claim}
\begin{proof}
	Let $E := \Delta(G)^k \cap A$. Then $E$ is a countable union of $k$-dimensional affine subspaces of $\bR^N$. Since $A \backslash \Delta(G)^{k-1}$ is an open subset of $A$, the preimage $\Omega_0 := u^{-1}(A \backslash \Delta(G)^{k-1})$ is an open subset of $\cN$. By assumption, $\Omega_0$ is non-empty. Furthermore, $u(\Omega_0) \subset A \cap (\Delta(G)^k \backslash \Delta(G)^{k-1})$. 
	
	Note that $A \cap (\Delta(G)^k \backslash \Delta(G)^{k-1}) = \bigsqcup_{i=1}^\infty E_i$, where each $E_i$ is an open subset of some $k$-dimensional affine subspace $L_i \subset \Delta(G)^k \cap A$. Thus, for some $L_i$, we have $u^{-1}(L_i) \cap \cN \neq \varnothing$. By \cref{claim:contain}, we conclude that $u(\cN) \subset L_i$. This proves the claim.
\end{proof}

\begin{claim}\label{claim:zero}
	Let $\cN$ be as in \cref{claim:smaller}. Then $u^{-1}(\Delta(G)^{k-1}) \cap \cN$ has zero Lebesgue measure.
\end{claim}
\begin{proof}
	By \cref{claim:smaller}, we have $u(\cN) \subset L\subset \Delta(G)^k$, where $L$ is a $k$-dimensional affine subspace of $\Delta(G)^k \cap A$. Note that $\Delta(G)^{k-1} \cap L=\cup_{j=1}^{\infty}H_j$, where each $H_j$ is an    affine hyperplanes  in $L$.  Fix some $H_j$ such that $u^{-1}(H_j)\cap \cN\neq\varnothing$.   We can choose a coordinate system $(x_1, \ldots, x_k)$ on $L$ such that $H_j$ is given by $(x_1 = 0)$. Write $u = (u_1, \ldots, u_k)$ with respect to this coordinate system. Then $$u^{-1}(H_j) \cap \cN = \{z \in \cN \mid u_1(z) = 0\}.$$ 
 	Note that $u_1|_\cN$ is not identically zero; otherwise, $u(\cN) \subset H \subset \Delta(G)^{k-1}$, contradicting the assumption. Let $E$ be the critical set of $u_1$. Then $E$ is a proper  analytic subset of $\cN$. Indeed, since $u$ is pluriharmonic by our assumption, it follows that $\partial\bar{\partial}u_1=0$, which implies that locally $u_1$ is the real part of a holomorphic function $h$. Thus, the critical set of $u$ coincides with the critical set of $h$, which is a proper analytic subset.
	
	On $\cN \backslash E$, the derivative $du_1$ is nowhere zero. Hence, $u_1^{-1}(0) \cap (\cN \backslash E)$ is a real $(2n-1)$-dimensional smooth submanifold of $\cN$, which has Lebesgue measure zero. Therefore, $u^{-1}(H_j)\subset E\cup u_1^{-1}(0) \cap (\cN \backslash E)$  has zero Lebesgue measure.  Since  $\cN\cap u^{-1}(\Delta(G)^{k-1})=\cup_{j=1}^\infty u^{-1}(H_j)$,  it follows that   $\cN\cap u^{-1}(\Delta(G)^{k-1})$ has zero Lebesgue measure. The claim is proved.
\end{proof}
We are now ready to prove the lemma.  

For any $x_0 \in \cR(u)$, we can choose a connected open neighborhood $\Omega_{x_0}$ of $x_0$ such that there exists an apartment $A$ of $\Delta(G)$ and a unique $k$ satisfying \( u(\Omega_{x_0}) \subset A \cap \Delta(G)^k \) and \( u(\Omega_{x_0}) \not\subset \Delta(G)^{k-1} \). We denote \( c(\Omega_{x_0}) := k \). By \cref{claim:smaller}, there exists a $k$-dimensional affine subspace \( L \) in \( A \cap \Delta(G)^k \) that contains the image \( u(\Omega_{x_0}) \).  

Let \( H \subset L \cap \Delta(G)^{k-1} \) be an affine hyperplane of \( L \). Then for any open subset \( U \subset \Omega_{x_0} \), we have \( u(U) \not\subset H \); otherwise, by \cref{claim:smaller}, it would follow that \( u(\Omega_{x_0}) \subset H \), contradicting the assumption that \( u(\Omega_{x_0}) \not\subset \Delta(G)^{k-1} \). This shows that, for any connected open subset \( U \subset \Omega_{x_0} \), we always have \( c(U) = c(\Omega_{x_0}) \).  

Thus, we can define a function \( c : \cR(u) \to \bZ \) by setting \( c(x_0) = c(\Omega_{x_0}) \). This function is well-defined and does not depend on the choice of \( \Omega_{x_0} \). Moreover, \( c \) is locally constant. Since \( \cR(u) \) is connected, it follows that \( c : \cR(u) \to \bZ \) is a constant function. Let \( k \) denote its value. Then \cref{item:subspace,item:notcontain} follow from \cref{claim:smaller} together with the definition of \( c : \cR(u) \to \bZ \).  

Now, for any \( x_0 \in \Omega \), let \( \Omega_{x_0} \) be a connected open neighborhood of \( x_0 \) as above. By \cref{claim:zero}, \( u^{-1}(\Delta(G)^{k-1}) \cap \Omega_{x_0} \) has zero Lebesgue measure. Since \( \Omega \) can be covered by countably many such open subsets, it follows that \( u^{-1}(\Delta(G)^{k-1}) \) has zero Lebesgue measure.  

Define \( \Omega^*(u) := \cR(u) \backslash u^{-1}(\Delta(G)^{k-1}) \). Since \( \Omega \backslash \cR(u) \) has Hausdorff codimension at least two, \( \Omega^*(u) \) has full measure in \( \Omega \). Note that \( u(\Omega^*(u)) \subset \Delta(G)^k \backslash \Delta(G)^{k-1} \), which is a disjoint union of \emph{facets} in \( \Delta(G)^k \) (cf. \cite[\S 2]{Rou09} for the definition). Each facet \( F \) in \( \Delta(G)^k \backslash \Delta(G)^{k-1} \) is an open subset of \( \Delta(G)^k \). Consequently, \( u^{-1}(F) \cap \cR(u) \) is an open subset of \( \Omega^*(u) \).  

Since each such \( F \) is contained in the closure of some chamber of \( \Delta(G) \), it follows that \( \Omega^*(u) \) satisfies the properties stated in the final assertion of the proposition.  

We complete  the proof of the proposition.   
\end{proof} 

\begin{rem}
	Note that \cref{omegastar} is of independent interest. For instance, it played a crucial role in \cite{DW24b} in the study of Koll\'ar’s conjecture on the positivity of the holomorphic Euler characteristic for varieties with large fundamental groups.
\end{rem}

\begin{proposition} \label{pretheorem1.1}
	Let $u_0,u_1:\Omega \rightarrow \Delta(G)$  be harmonic maps from a bounded Riemannian domain. If  $
	d(u_0,u_1)=c$ for some constant $c\geq 0$ and $|\nabla u_0|^2=|\nabla u_1|^2$, then for almost every point  $x\in \Omega$, there exists  $r>0$ satisfying the following:
	\begin{thmlist}
		\item There is a $N$-flat $A$ containing both $u_0(B_r(x))$ and $u_1(B_r(x))$;
		\item  If we fix an isometry $\nu:A\to \bR^N$, then $\nu\circ u_0:B_r(x)\to \bR^N$ is a translation of $\nu\circ u_1:B_r(x) \to \bR^N$.  
	\end{thmlist}
\end{proposition}

\begin{proof}
	For $i=0,1$, let $\Omega^*(u_i)$ be the open set of full measure as in Lemma~\ref{omegastar}.   Thus, $\Omega^*(u_0) \cap \Omega^*(u_1)$ is of full measure.  Lemma~\ref{omegastar} implies that, for any $x_0 \in \Omega^*(u_0) \cap \Omega^*(u_1)$, there exists $r>0$ and a chamber $C_i$ such that $u_i(B_r(x_0)) \subset C_i$ for $i=0,1$.  Let $A$ be $N$-flat containing chambers $C_0$ and $C_1$ and $\nu: A \rightarrow \mathbb R^N$ be an isometry.  Thus, $v \circ u_0$ and $v \circ u_1$ are harmonic maps into $\mathbb R^N$.  The assumption that $d(u_0, u_1)=c$ implies that $|\nu \circ u_0(x)-\nu \circ u_1(x)|=c$.  Thus, $0=\Delta |\nu \circ u_0-\nu \circ u_1|^2= 2\left|\nabla  (\nu \circ u_0-\nu \circ u_1) \right|^2$ which implies $\nu \circ u_0$ is a translation of $\nu \circ u_1$. 
\end{proof} 
 We are able to prove \cref{main:unicity}.
\begin{proof}[Proof of \cref{main:unicity}]
	The assertion follows immediately from Proposition~\ref{prop:uv} and Proposition~\ref{pretheorem1.1}.
\end{proof}
\begin{proof}[Proof of \cref{item:unicity}]
	By \cref{main:unicity}, there exists a dense open subset $\widetilde{X}^\circ\subset \widetilde{X}$ of  full Lebesgue measure such that,  for any $x\in \widetilde{X}^\circ$, 
	\begin{enumerate}[label=(\alph*)]
		\item there exists an open neighborhood $\Omega$ of $x$ and an apartment $A$ of $\Delta(G)$ such that
		$\tilde{u}_i(\Omega)\subset A$ for $i=0,1$;
		\item the map $\tilde{u}_0|_{\Omega}:\Omega\to A$ is a translate of $\tilde{u}_1|_{\Omega}:\Omega\to A$ 
	\end{enumerate}
	By the construction  in \cite{BDDM}, the  multivalued 1-forms $\eta_i $  induced $\tilde{u}_i$  for $i=0,1$ are equal over $\widetilde{X}^\circ$, and splitting over $\widetilde{X}^\circ$.  By \cref{def:multivalued}, we conclude that $\eta_1=\eta_2$ over the entire $X$. The claim is proved. 
\end{proof}
\section{On the singular set of harmonic maps into Euclidean buildings}\label{sec:singular}
In this section, we apply \cref{lem:5} and the results from \cref{subsec:unicity}
to prove \cref{item:singular}, following the idea by Eyssidieux in \cite[Proposition 1.3.3]{Eys04}.
\begin{thm}[=\cref{item:singular}]\label{thm:singular}
	Let $X$, $\varrho$, $G$ and $\tilde{u}$ be as in \cref{main:harmonic}.   Then the singular set $\cS(u)$  defined in \cref{def:sing} is contained in a proper Zariski closed subset of $X$.
\end{thm}
\begin{proof}
	We assume that $\tilde{u}$ is non-constant. 	We shall use the notions in \cref{subsec:unicity} with $\Omega$ being $\widetilde{X}$. Let $M$ be the positive integer defined in \cref{dfn:M}.  Let $\eta$ be the logarithmic multivalued 1-form induced by  $\tilde{u}$ defined in \cref{lem:5}.  Let   $Z_\eta=\sum_{i=1}^{k}n_iZ_i$ be the formal sum  corresponding to $\eta$  defined in \cref{def:multivalued}, where each $Z_i$ is an irreducible closed subvariety of $E$ such that the natural map $Z_i\to X$ is surjective and finite. Let $|Z_\eta|=\cup_{i=1}^{k}Z_i\subset \Omega_{ \overline{X}}(\log \Sigma)$ be the support of $Z_\eta$. 
	
	Let $M$ be the positive integer defined in \cref{dfn:M}.	Consider the holomorphic  bundle  $E:=\Omega^M_{ \overline{X}}(\log \Sigma)$ on $\overline{X}$.  By \cref{item:p}, $\eta$ induces a multivalued section $\eta^M$ of $E$.     
	Let $|Z_{\eta^M}|\subset E$ be the support of the formal sum $Z_{\eta^M}$ induced by    ${\eta^M}$ defined in \cref{def:multivalued}.   Let $\overline{X}^\circ$ be the set of points $x$ in $\overline{X}$ such that  $|Z_{\eta^M}|_x\not\subset \{0\}$. We shall prove that the Zariski open subset $\overline{X}^\circ$ is dense in $\overline{X}$.
	
	By our definition of $M$ in \cref{dfn:M} and \cref{lem:compl}, for any point $x_0\in \tilde{\cR}(u)$, there exists $r>0$ such that:
	\begin{enumerate}[label=(\arabic*)]
		\item   for some connected component $\Omega$ of $\pi_X^{-1}(B_r(x_0))$, $\pi_X|_{\Omega}:\Omega\to B_r(x_0)$ is an isomorphism, where $B_r(x_0)$  is the geodesic ball centered at $x_0$ of radius $r$. 
		\item  We have the decomposition 	\begin{equation*}  
			\tilde{u}\circ (\pi_{X}^{-1}|_{\Omega})|_{B_r(x_0)}=(u^1,u^2): B_r(x_0) \rightarrow \bR^M \times \{P_0\},
		\end{equation*}  
		where  $u^1$ is a harmonic map with rank $M$ at each point of $B_r(x_0)$.
	\end{enumerate}  By \cref{item:ortho}, $\eta$ is represented by  $\partial u^1$ up to some orthogonal transformation and rescaling. It follows that $|Z_{\eta^M}|_x$ is not $\{0\}$ for every $x\in B_r(x_0)$. Hence we have
\begin{align}\label{eq:inclusion}
	\tilde{\cR}(u)\subset \overline{X}^\circ,
\end{align}
which implies that $ \overline{X}^\circ$ is  non-empty.  Since $\overline{X}^\circ$ is Zariski open in $\overline{X}$, it follows that $ {X}^\circ:=X\cap \overline{X}^\circ$ is  a dense  and Zariski open subset of $X$. The theorem follows from  \cref{lem:compl} together with  \cref{claim:ord2} below.
\end{proof}

\begin{lem}\label{claim:ord2}
	We have $ \tilde{\cS}(u) =  X\backslash {X}^\circ$. 
\end{lem}
\begin{proof}
	Let $x_0\in X$.  If ${\rm Ord}^u(x_0) >1$, then $|\nabla u|^2(x_0)=0$ by \cref{ord1}.  If ${\rm Ord}^u(x_0)=1$, then we  apply \Cref{item:splitting}  above \eqref{prod}.  Thus, in either case,  there exists $r>0$ and an $F$-connected complex $\cF$ such that 
	\begin{enumerate}[label=(\alph*)]
		\item   for some connected component $\Omega$ of $\pi_X^{-1}(B_r(x_0))$, $\pi_X|_{\Omega}:\Omega\to B_r(x_0)$ is an isomorphism.
		\item  We have 	\begin{equation}  \label{product2}
			\tilde{u}\circ (\pi_{X}^{-1}|_{\Omega})|_{B_r(x_0)}=(u^1,u^2): B_r(x_0) \rightarrow   \bR^k \times \cF,
		\end{equation}
		such that $u^1: B_{r}(x_0) \rightarrow \bR^k$ is a smooth pluriharmonic map with rank at each point of $B_\sigma(x_0)$ equal to  $k$ (see the proof of \cref{lem:compl}) and  $u^2:B_{r}(x_0) \rightarrow \cF$ is a pluriharmonic map with ${\rm Ord}^{u^2}(x_0) \geq 1+\ep$ for some $\ep>0$ and $|\nabla u^2|(x_0) =0$ by \cref{ord1}.  Here, we are using the following convention:  If $k=M$, then $u^2$ is a constant map, and if $k=0$, then $(u^1,u^2)=u^2$. 
	\end{enumerate}
 	Note that $\cF$ has an Euclidean building structure.  	 By the proof of \cref{item:extend} and \cref{lem:5},    the pluriharmonic map $u^2$ in \eqref{product2} induces a multivalued   1-form $\psi_0$ on $B_r(x_0)$ satisfying the properties in \cref{lem:5}.  Then for each $x_1\in B_r(x_0)\cap \cR(u)$, it has a neighborhood $\Omega_{x_1}$  over  which the multivalued   1-form $\psi$ are given by holomorphic 1-forms $\sqcup_{i=1}^{\ell}\{\psi_{i,1},\ldots,\psi_{i,N-k}\}$, that is the  partition of $\psi_0$ in \cref{item:ortho}.  
	By \eqref{eq:energy3}, one has 
	\begin{align}\label{eq:density2}
		|\nabla u^2|^2=2\sum_{i=1}^{\ell}\sum_{j=1}^{N-k}|\psi_{i,j}|^2.
	\end{align} 
	We define $\psi_{i,N-k+j}:=\frac{1}{\sqrt{\ell}}\partial u_{j}^1$ for each $i\in \{1,\ldots,\ell\}$ and $j\in \{1,\ldots,k\}$.  Therefore,  $\sqcup_{i=1}^{\ell}\{\psi_{i,1},\ldots,\psi_{i,N}\}$ is a multivalued 1-form associated with $(u^1,u^2)$ defined in \cref{lem:5}. 
	
	We  can shrink $\Omega_{x_1}$ such that $\eta$ is given by  holomorphic 1-forms $\sqcup_{i=1}^{\ell'}\{\omega_{i,1},\ldots,\omega_{i,N}\}$ on $\Omega_{x_1}$, that is the  partition of $\eta$ in \cref{item:ortho}.  Hence, by \eqref{eq:orth} and \eqref{eq:orth2},  for each $i$ and $j$, there exists a constant matrix $M_{i,j}\in {\rm O}(N,\bR)$ such that
	\begin{align}\label{eq:ortho}
		\begin{bmatrix}
			\omega_{i,1}, & \ldots, & \omega_{i,N}
		\end{bmatrix} =	 \begin{bmatrix}
			\psi_{j,1}, & \cdots, &\psi_{j,N}
		\end{bmatrix} \cdot M_{i,j} \cdot  \frac{\sqrt{\ell'}}{\sqrt{\ell}}.
	\end{align} 
	By \eqref{eq:orth} and the definition of $\eta^M$ in \cref{item:p}, over $\Omega_{x_1}$,  the multivalued section ${\eta^M}$ of $E$ is given by 
	$$
	\sqcup_{j=1}^{\ell'}\{\pm \omega_{j,i_1}\wedge\cdots\wedge \omega_{j,i_M} \}_{1\leq i_1<\cdots<i_M\leq N}.
	$$
On the other hand, 	by \cref{item:p}, $(u^1,u^2)$ induces another multivalued section of  $E|_{B_r(x_0)}$, which is locally represented by 
	$$
	\sqcup_{j=1}^{\ell}\{\pm \psi_{j,i_1}\wedge\cdots\wedge \psi_{j,i_M} \}_{1\leq i_1<\cdots<i_M\leq N}.
	$$
	For notational simplicity, for each $I=(i_1,\ldots,i_M)\subset \{1,\ldots,N\}$ with $1\leq i_1<\cdots<i_M\leq N$, we  write
	$$
	\omega_{j,I}:=\omega_{j,i_1}\wedge\cdots\wedge \omega_{j,i_M}, \quad \forall j\in \{1,\ldots,\ell'\},
	$$
	and
	$$
	\psi_{j,I}:=\psi_{j,i_1}\wedge\cdots\wedge \psi_{j,i_M},\quad \forall j\in \{1,\ldots,\ell\}.
	$$
	Therefore, by \eqref{eq:ortho} there exists a constant matrix of $\tilde{M}_{i,j}\in {\rm O}(\binom{N}{M},\bR)$  such that
	\begin{align*} 
		[\omega_{j,I}]_{1\leq i_1<\cdots<i_M\leq N}  =[\psi_{i,I}]_{1\leq i_1<\cdots<i_M\leq N}\cdot \tilde{M}_{i,j}(\frac{\sqrt{\ell'}}{\sqrt{\ell}})^M.
	\end{align*} 
	Thus, we have the following equality, which holds over the entire $\cR(u)\cap B_r(x_0)$:
	\begin{align}\label{eq:sum}
		\sum_{j=1}^{\ell'}\sum_{1\leq i_1<\cdots<i_M\leq N}|\omega_{j,I}|_{h_E}^2=  \frac{(\ell')^{M+1}}{\ell^M}\sum_{1\leq i_1<\cdots<i_M\leq N}|\psi_{1,I}|_{h_E}^2,
	\end{align}  
	Note that there exists a constant $C>1$ such that  
	\begin{align}\label{eq:sum2}|\psi_{i,N-k+j}(x)|=|\frac{1}{\sqrt{\ell}}\partial u_{j}^1|\leq C, \quad \forall\  x\in B_{r}(x_0)
	\end{align} for each $i\in \{1,\ldots,\ell\}$ and $j\in \{1,\ldots,k\}$. 
	
	We now consider the cases of $x_0 \in \tilde \cS(u)$ and $x_0 \in \tilde \cR(u)$ separately:
	\begin{enumerate}[label=(\alph*)]
		\item
		If $x_0 \in \tilde \cS(u)$, then $M>k$.  By \eqref{eq:density2} and \eqref{eq:sum2},  for    any $x\in   \Omega_{x_1}$, we have
		\begin{align}\label{eq:tends}
			| \psi_{1,I}|^2\leq   C^{2k} |\nabla u^2|^{2\lambda(I)},
		\end{align} 
		where $\lambda(I)$ denotes the cardinality of $I\cap \{1,\ldots,N-k\}$,   that is a positive integer. Recall that $|\nabla u^2(x_0)|^2=0$. Since $C$ is a constant independent of $x_1\in B_r(x_0)\cap\cR(u)$,  it then follows from  \eqref{eq:sum} and \eqref{eq:tends}  that
		$$
		\lim\limits_{x\in \cR(u),x\to x_0}\sum_{j=1}^{\ell'}\sum_{1\leq i_1<\cdots<i_M\leq N}|\omega_{j,I}|^2(x)=0. 
		$$
		Since the multivalued section ${\eta^M}$  on $E$ is locally represented  by   
		$
		\sqcup_{j=1}^{\ell'}\{\pm \omega_{j,I} \}_{1\leq i_1<\cdots<i_M\leq m}
		$, it follows that 
		that $|Z_{\eta^M}|_{x_0}\subset \{0\}$.  In other words, $x_0 \notin  {X}^\circ$.
		\item
		If $x_0 \in \tilde \cR(u)$,  by \eqref{eq:inclusion}, we have $x_0\in X^\circ$.  
	\end{enumerate}
	In conclusion, we have $\tilde{\cR}(u)=X^\circ$. The lemma is proved. 
\end{proof}


\begin{thebibliography}{BDDM22}
 	\newcommand{\enquote}[1]{``#1''}
 	\providecommand{\url}[1]{\texttt{#1}}
 	\providecommand{\urlprefix}{URL }
 	\providecommand{\bibinfo}[2]{#2}
 	\providecommand{\eprint}[2][]{\url{#2}}
 	
 	\bibitem[BDDM22]{BDDM}
 	\bibinfo{author}{D.~{Brotbek}}, \bibinfo{author}{G.~{Daskalopoulos}},
 	\bibinfo{author}{Y.~{Deng}} \& \bibinfo{author}{C.~{Mese}}.
 	\newblock \enquote{\bibinfo{title}{{Pluriharmonic maps into buildings and
 				symmetric differentials}}}.
 	\newblock \emph{\bibinfo{journal}{arXiv e-prints}},
 	\textbf{(\bibinfo{year}{2022})}:\bibinfo{eid}{arXiv:2206.11835}.
 	\newblock \eprint{2206.11835},
 	\urlprefix\url{http://dx.doi.org/10.48550/arXiv.2206.11835}.
 	
 	\bibitem[BDM24]{BDM24}
 	\bibinfo{author}{C.~{Breiner}}, \bibinfo{author}{B.~K. {Dees}} \&
 	\bibinfo{author}{C.~{Mese}}.
 	\newblock \enquote{\bibinfo{title}{{Harmonic Maps into Euclidean Buildings and
 				Non-Archimedean Superrigidity}}}.
 	\newblock \emph{\bibinfo{journal}{arXiv e-prints}},
 	\textbf{(\bibinfo{year}{2024})}:\bibinfo{eid}{arXiv:2408.02783}.
 	\newblock \eprint{2408.02783},
 	\urlprefix\url{http://dx.doi.org/10.48550/arXiv.2408.02783}.
 	
 	\bibitem[BH99]{bridson-haefliger}
 	\bibinfo{author}{M.~R. Bridson} \& \bibinfo{author}{A.~Haefliger}.
 	\newblock \emph{\bibinfo{title}{Metric spaces of non-positive curvature}}, vol.
 	\bibinfo{volume}{319} of \emph{\bibinfo{series}{Grundlehren Math. Wiss.}}
 	\newblock \bibinfo{publisher}{Berlin: Springer} (\bibinfo{year}{1999}).
 	
 	\bibitem[CDY22]{CDY22}
 	\bibinfo{author}{B.~{Cadorel}}, \bibinfo{author}{Y.~{Deng}} \&
 	\bibinfo{author}{K.~{Yamanoi}}.
 	\newblock \enquote{\bibinfo{title}{{Hyperbolicity and fundamental groups of
 				complex quasi-projective varieties}}}.
 	\newblock \emph{\bibinfo{journal}{arXiv e-prints}},
 	\textbf{(\bibinfo{year}{2022})}:\bibinfo{eid}{arXiv:2212.12225}.
 	\newblock \eprint{2212.12225},
 	\urlprefix\url{http://dx.doi.org/10.48550/arXiv.2212.12225}.
 	
 	\bibitem[Cor88]{Cor88}
 	\bibinfo{author}{K.~Corlette}.
 	\newblock \enquote{\bibinfo{title}{Flat {$G$}-bundles with canonical metrics}}.
 	\newblock \emph{\bibinfo{journal}{J. Differential Geom.}},
 	\textbf{\bibinfo{volume}{28}(\bibinfo{year}{1988})(\bibinfo{number}{3})}:\bibinfo{pages}{361--382}.
 	\newblock \urlprefix\url{http://projecteuclid.org/euclid.jdg/1214442469}.
 	
 	\bibitem[CS08]{CS08}
 	\bibinfo{author}{K.~Corlette} \& \bibinfo{author}{C.~Simpson}.
 	\newblock \enquote{\bibinfo{title}{On the classification of rank-two
 			representations of quasiprojective fundamental groups}}.
 	\newblock \emph{\bibinfo{journal}{Compos. Math.}},
 	\textbf{\bibinfo{volume}{144}(\bibinfo{year}{2008})(\bibinfo{number}{5})}:\bibinfo{pages}{1271--1331}.
 	\newblock \urlprefix\url{http://dx.doi.org/10.1112/S0010437X08003618}.
 	
 	\bibitem[CT89]{carlson-toledo}
 	\bibinfo{author}{J.~A. Carlson} \& \bibinfo{author}{D.~Toledo}.
 	\newblock \enquote{\bibinfo{title}{Harmonic mappings of {K}\"ahler manifolds to
 			locally symmetric spaces}}.
 	\newblock \emph{\bibinfo{journal}{Inst. Hautes \'Etudes Sci. Publ. Math.}},
 	\textbf{(\bibinfo{year}{1989})(\bibinfo{number}{69})}:\bibinfo{pages}{173--201}.
 	\newblock \urlprefix\url{http://www.numdam.org/item?id=PMIHES_1989__69__173_0}.
 	
 	\bibitem[Dee22]{Dee22}
 	\bibinfo{author}{B.~K. Dees}.
 	\newblock \enquote{\bibinfo{title}{Rectifiability of the singular set of
 			harmonic maps into buildings}}.
 	\newblock \emph{\bibinfo{journal}{J. Geom. Anal.}},
 	\textbf{\bibinfo{volume}{32}(\bibinfo{year}{2022})(\bibinfo{number}{7})}:\bibinfo{pages}{Paper
 		No. 205, 57}.
 	\newblock \urlprefix\url{http://dx.doi.org/10.1007/s12220-022-00943-x}.
 	
 	\bibitem[DM23a]{DMrs}
 	\bibinfo{author}{G.~Daskalopoulos} \& \bibinfo{author}{C.~Mese}.
 	\newblock \enquote{\bibinfo{title}{Infinite energy harmonic maps from {R}iemann
 			surfaces to {$\rm CAT(0)$} spaces}}.
 	\newblock \emph{\bibinfo{journal}{J. Geom. Anal.}},
 	\textbf{\bibinfo{volume}{33}(\bibinfo{year}{2023})(\bibinfo{number}{10})}:\bibinfo{pages}{Paper
 		No. 337, 23}.
 	\newblock \urlprefix\url{http://dx.doi.org/10.1007/s12220-023-01382-y}.
 	
 	\bibitem[DM23b]{DMunique}
 	---{}---{}---.
 	\newblock \enquote{\bibinfo{title}{Uniqueness of equivariant harmonic maps to
 			symmetric spaces and buildings}}.
 	\newblock \emph{\bibinfo{journal}{Math. Res. Lett.}},
 	\textbf{\bibinfo{volume}{30}(\bibinfo{year}{2023})(\bibinfo{number}{6})}:\bibinfo{pages}{1639--1655}.
 	
 	\bibitem[DM24]{DMks}
 	---{}---{}---.
 	\newblock \enquote{\bibinfo{title}{Infinite energy harmonic maps from
 			quasi-compact {K}\"ahler surfaces}}.
 	\newblock \emph{\bibinfo{journal}{Adv. Nonlinear Stud.}},
 	\textbf{\bibinfo{volume}{24}(\bibinfo{year}{2024})(\bibinfo{number}{1})}:\bibinfo{pages}{103--140}.
 	\newblock \urlprefix\url{http://dx.doi.org/10.1515/ans-2023-0122}.
 	
 	\bibitem[DMW24]{DMW24}
 	\bibinfo{author}{Y.~{Deng}}, \bibinfo{author}{C.~{Mese}} \&
 	\bibinfo{author}{B.~{Wang}}.
 	\newblock \enquote{\bibinfo{title}{{Deformation Openness of Big Fundamental
 				Groups and Applications}}}.
 	\newblock \emph{\bibinfo{journal}{arXiv e-prints}},
 	\textbf{(\bibinfo{year}{2024})}:\bibinfo{eid}{arXiv:2412.08636}.
 	\newblock \eprint{2405.08636},
 	\urlprefix\url{http://dx.doi.org/10.48550/arXiv.2412.08636}.
 	
 	\bibitem[DW24a]{DW24b}
 	\bibinfo{author}{Y.~{Deng}} \& \bibinfo{author}{B.~{Wang}}.
 	\newblock \enquote{\bibinfo{title}{{$L^2$-vanishing theorem and a conjecture of
 				Koll{\'a}r}}}.
 	\newblock \emph{\bibinfo{journal}{arXiv e-prints}},
 	\textbf{(\bibinfo{year}{2024})}:\bibinfo{eid}{arXiv:2409.11399}.
 	\newblock \eprint{2409.11399},
 	\urlprefix\url{http://dx.doi.org/10.48550/arXiv.2409.11399}.
 	
 	\bibitem[DW24b]{DW24}
 	---{}---{}---.
 	\newblock \enquote{\bibinfo{title}{{Linear Chern-Hopf-Thurston conjecture}}}.
 	\newblock \emph{\bibinfo{journal}{arXiv e-prints}},
 	\textbf{(\bibinfo{year}{2024})}:\bibinfo{eid}{arXiv:2405.12012}.
 	\newblock \eprint{2405.12012},
 	\urlprefix\url{http://dx.doi.org/10.48550/arXiv.2405.12012}.
 	
 	\bibitem[Eyr04]{Eyr04}
 	\bibinfo{author}{C.~Eyral}.
 	\newblock \enquote{\bibinfo{title}{Topology of quasi-projective varieties and
 			{Lefschetz} theory}}.
 	\newblock \emph{\bibinfo{journal}{N. Z. J. Math.}},
 	\textbf{\bibinfo{volume}{33}(\bibinfo{year}{2004})(\bibinfo{number}{1})}:\bibinfo{pages}{63--81}.
 	
 	\bibitem[{Eys}04]{Eys04}
 	\bibinfo{author}{P.~{Eyssidieux}}.
 	\newblock \enquote{\bibinfo{title}{{Sur la convexit\'e holomorphe des
 				rev\^etements lin\'eaires r\'eductifs d'une vari\'et\'e projective
 				alg\'ebrique complexe}}}.
 	\newblock \emph{\bibinfo{journal}{{Invent. Math.}}},
 	\textbf{\bibinfo{volume}{156}(\bibinfo{year}{2004})(\bibinfo{number}{3})}:\bibinfo{pages}{503--564}.
 	\newblock \urlprefix\url{http://dx.doi.org/10.1007/s00222-003-0345-0}.
 	
 	\bibitem[GS92]{GS92}
 	\bibinfo{author}{M.~{Gromov}} \& \bibinfo{author}{R.~{Schoen}}.
 	\newblock \enquote{\bibinfo{title}{{Harmonic maps into singular spaces and
 				\(p\)-adic superrigidity for lattices in groups of rank one}}}.
 	\newblock \emph{\bibinfo{journal}{{Publ. Math., Inst. Hautes \'Etud. Sci.}}},
 	\textbf{\bibinfo{volume}{76}(\bibinfo{year}{1992})}:\bibinfo{pages}{165--246}.
 	\newblock \urlprefix\url{http://dx.doi.org/10.1007/BF02699433}.
 	
 	\bibitem[JY91]{jost-yau}
 	\bibinfo{author}{J.~Jost} \& \bibinfo{author}{S.-T. Yau}.
 	\newblock \enquote{\bibinfo{title}{Harmonic maps and {K}\"ahler geometry}}.
 	\newblock \enquote{\bibinfo{booktitle}{Prospects in complex geometry ({K}atata
 			and {K}yoto, 1989)}}, vol. \bibinfo{volume}{1468} of
 	\emph{\bibinfo{series}{Lecture Notes in Math.}}, \bibinfo{pages}{340--370}.
 	\bibinfo{publisher}{Springer, Berlin}
 	(\bibinfo{year}{1991}):\hspace{0pt}\urlprefix\url{http://dx.doi.org/10.1007/BFb0086200}.
 	
 	\bibitem[Kli13]{Kli13}
 	\bibinfo{author}{B.~Klingler}.
 	\newblock \enquote{\bibinfo{title}{Symmetric differentials, {K{\"a}hler} groups
 			and ball quotients}}.
 	\newblock \emph{\bibinfo{journal}{Invent. Math.}},
 	\textbf{\bibinfo{volume}{192}(\bibinfo{year}{2013})(\bibinfo{number}{2})}:\bibinfo{pages}{257--286}.
 	\newblock \urlprefix\url{http://dx.doi.org/10.1007/s00222-012-0411-6}.
 	
 	\bibitem[KP23]{KP23}
 	\bibinfo{author}{T.~Kaletha} \& \bibinfo{author}{G.~Prasad}.
 	\newblock \emph{\bibinfo{title}{Bruhat-{T}its theory---a new approach}},
 	vol.~\bibinfo{volume}{44} of \emph{\bibinfo{series}{New Mathematical
 			Monographs}}.
 	\newblock \bibinfo{publisher}{Cambridge University Press, Cambridge}
 	(\bibinfo{year}{2023}).
 	
 	\bibitem[KS93]{KS}
 	\bibinfo{author}{N.~J. {Korevaar}} \& \bibinfo{author}{R.~M. {Schoen}}.
 	\newblock \enquote{\bibinfo{title}{{Sobolev spaces and harmonic maps for metric
 				space targets}}}.
 	\newblock \emph{\bibinfo{journal}{{Commun. Anal. Geom.}}},
 	\textbf{\bibinfo{volume}{1}(\bibinfo{year}{1993})(\bibinfo{number}{4})}:\bibinfo{pages}{561--659}.
 	\newblock \urlprefix\url{http://dx.doi.org/10.4310/CAG.1993.v1.n4.a4}.
 	
 	\bibitem[Moc07]{Moc07}
 	\bibinfo{author}{T.~Mochizuki}.
 	\newblock \enquote{\bibinfo{title}{Asymptotic behaviour of tame harmonic
 			bundles and an application to pure twistor {$D$}-modules. {I}}}.
 	\newblock \emph{\bibinfo{journal}{Mem. Amer. Math. Soc.}},
 	\textbf{\bibinfo{volume}{185}(\bibinfo{year}{2007})(\bibinfo{number}{869})}:\bibinfo{pages}{xii+324}.
 	\newblock \urlprefix\url{http://dx.doi.org/10.1090/memo/0869}.
 	
 	\bibitem[Rou09]{Rou09}
 	\bibinfo{author}{G.~Rousseau}.
 	\newblock \enquote{\bibinfo{title}{Euclidean buildings}}.
 	\newblock \enquote{\bibinfo{booktitle}{G\'eom\'etries \`a courbure n\'egative
 			ou nulle, groupes discrets et rigidit\'es}}, \bibinfo{pages}{77--116}.
 	\bibinfo{publisher}{Paris: Soci{\'e}t{\'e} Math{\'e}matique de France (SMF)}
 	(\bibinfo{year}{2009}):\hspace{0pt}.
 	
 	\bibitem[Sam85]{Sam85}
 	\bibinfo{author}{J.~H. Sampson}.
 	\newblock \enquote{\bibinfo{title}{Harmonic maps in {K}\"{a}hler geometry}}.
 	\newblock \enquote{\bibinfo{booktitle}{Harmonic mappings and minimal immersions
 			({M}ontecatini, 1984)}}, vol. \bibinfo{volume}{1161} of
 	\emph{\bibinfo{series}{Lecture Notes in Math.}}, \bibinfo{pages}{193--205}.
 	\bibinfo{publisher}{Springer, Berlin}
 	(\bibinfo{year}{1985}):\hspace{0pt}\urlprefix\url{http://dx.doi.org/10.1007/BFb0075138}.
 	
 	\bibitem[Shi68]{Shi68}
 	\bibinfo{author}{B.~Shiffman}.
 	\newblock \enquote{\bibinfo{title}{On the removal of singularities of analytic
 			sets}}.
 	\newblock \emph{\bibinfo{journal}{Mich. Math. J.}},
 	\textbf{\bibinfo{volume}{15}(\bibinfo{year}{1968})}:\bibinfo{pages}{111--120}.
 	\newblock \urlprefix\url{http://dx.doi.org/10.1307/mmj/1028999912}.
 	
 	\bibitem[Siu80]{siu}
 	\bibinfo{author}{Y.~T. Siu}.
 	\newblock \enquote{\bibinfo{title}{The complex-analyticity of harmonic maps and
 			the strong rigidity of compact {K}\"ahler manifolds}}.
 	\newblock \emph{\bibinfo{journal}{Ann. of Math. (2)}},
 	\textbf{\bibinfo{volume}{112}(\bibinfo{year}{1980})(\bibinfo{number}{1})}:\bibinfo{pages}{73--111}.
 	\newblock \urlprefix\url{http://dx.doi.org/10.2307/1971321}.
 	
 \end{thebibliography}

\end{document}